\documentclass[10pt]{article}
\usepackage{amsmath,amssymb,graphicx,psfrag}
\usepackage{subcaption}
\usepackage{enumerate}
\usepackage{xmpmulti}
\usepackage{amssymb}
\usepackage{rotating}
\usepackage{amsmath}
\usepackage{amsthm}
\usepackage{amscd}
\usepackage[square, comma, sort&compress, numbers]{natbib} 
\usepackage{epsfig}
\usepackage{color}
\usepackage[latin5]{inputenc}
\usepackage{psfrag}
\usepackage{graphicx}
\newcommand\Z{\mathbb{Z}}

\newcommand\R{\mathbb{R}}

\newcommand\cA{{\mathcal A}}
\newcommand\cF{{\mathcal F}}
\newcommand\cS{{\mathcal S}}

\newcommand\cL{{\mathcal L}}

\newcommand{\tropical}[1]{\left[\,#1\,\right]}
\newcommand{\I}{^{-1}}

\DeclareMathOperator{\rank}{rank}

\DeclareMathOperator{\MF}{\mathcal{MF}}

\DeclareMathOperator{\MCG}{MCG}

\allowdisplaybreaks[1]

\newtheorem{thm}{Theorem}[section]
\newtheorem{lem}[thm]{Lemma}

\newtheorem*{prooftheorem31}{Proof of Theorem 1.5}
\newtheorem*{prooftheorem34}{Proof of Theorem 1.8}
\newtheorem*{prooflemma33}{Proof of Lemma 1.7}
\newtheorem*{prooftheorem35}{Proof of Theorem 1.9}

\newtheorem{corollary}[thm]{Corollary}

\theoremstyle{definition}		% use "definition-style" font for the rest
\newtheorem{mydef}[thm]{Definition}
\newtheorem{example}[thm]{Example}

\newtheorem{mydefs}[thm]{Definitions}

\newtheorem{remark}[thm]{Remark}

\newtheorem{quest}[thm]{Question}

\newtheorem*{proof-sketch}{Sketch of Proof}

\begin{document}

\begin{center}
\large{Dynnikov and train track transition matrices of pseudo-Anosov braids}\\

\vspace{3mm}

\small{S. \"OYK\"U YURTTA\c S}
\end{center}

\begin{abstract}
We compare the spectra  of \emph{Dynnikov matrices} with the spectra of the \emph{train track transition matrices} of a given pseudo-Anosov braid on the finitely punctured disk, and show that these matrices are isospectral up to roots of unity and zeros under some particular conditions.  It is shown, via examples, that Dynnikov matrices are much easier to compute than transition matrices, and so yield data that was previously inaccessible.
\end{abstract}

\section{Introduction and statement of results}

 The Nielsen-Thurston classification theorem states that every homeomorphism of a compact orientable surface is isotopic to either a finite order or a pseudo-Anosov or a reducible \cite{W88} homeomorphism. If some iterate of a homeomorphism is the identity, it is called finite order. If a homeomorphism preserves a transverse pair of measured foliations \cite{W88, FLP79} on the
surface, stretching one foliation uniformly by a real number $\lambda>1$ (\emph{dilatation}) and contracting
the other uniformly by $1/\lambda$, then it is called pseudo -Anosov. If a homeomorphism
preserves a collection of mutually disjoint essential simple closed curves (reducing
curves), it is called reducible.

			The usual way to study the dynamics of an isotopy class of surface homeomorphisms is to use Thurston's train tracks \cite{bh95,los}. In \cite{bh95}, the algorithm starts with a graph $G$ which is a spine of the surface and the isotopy class is represented by a graph map. The algorithm repeatedly modifies $G$ and the associated graph map until it either finds an explicit reducing curve for the isotopy class , or a graph map which is the simplest possible. If the isotopy class is pseudo\,-Anosov, this simplest graph map can be used to construct \emph{a train track} and \emph{train track transition matrix} from which the dilatation of the isotopy class, the singularity structure of the invariant foliations and the periodic orbit structure are obtained. However, for isotopy classes with high dilatation the lengths of the image edge paths (represented by words whose letters label the edges) of the train track  are so long such that even a computer cannot store them.

		 In this paper we introduce an alternative matrix for a given pseudo-Anosov isotopy class  which can give us the same dynamical information as the train track transition matrix does. That is,  we present \emph{Dynnikov matrices} as a new tool to study the dynamics of pseudo-Anosov isotopy classes on the $n$-punctured disk  $D_n$ ($n\geq 3$). These matrices are much easier to compute than computing train-track transition matrices (Section \ref{comparemethods} gives an example to contrast the computation of a Dynnikov matrix with that of the train track transition matrix of a given pseudo-Anosov isotopy class on $D_4$). Roughly speaking, a \emph{Dynnikov matrix} is an integer matrix which describes the action of a given pseudo-Anosov isotopy class in a neighbourhood of its invariant unstable measured foliation in terms of \emph{Dynnikov's coordinates} \cite{D02, DDRW02, or08, paper1, paper2}  on the space of projective measured foliations on $D_n$. The dilatation of a  given a pseudo-Anosov isotopy class equals the spectral radius of its Dynnikov matrix. Making use of this fact, in \cite{paper1} we gave an alternative approach to compute the dilatation of each member of an infinite family of pseudo-Anosov isotopy classes on $D_n$ using Dynnikov matrices. The aim of this paper is to show that it is not only the dilatation but the whole set of eigenvalues (up to roots of unity) that train track transition and Dynnikov matrices share.  This will yield a computationally much more efficient way to study the dynamics of pseudo-Anosov isotopy classes on $D_n$. Let us briefly explain the idea behind our method and then describe the Dynnikov coordinate system. 
		
			Let $M$ be a compact, orientable surface (perhaps with boundary) with negative Euler characteristic. Let $\mathcal{MF}(M)$ be the space of measured foliations on $M$ and $\mathcal{PMF}(M)$ be the corresponding space of projective measured foliations.

The method developed in this paper lies in Thurston's seminal paper  on the geometry and dynamics of surface automorphisms \cite{W88} and  builds on more recent work of Moussafir \cite{M06}. The Teichm\"uller space $\mathcal{T}(M)$ of $M$ is an open ball and $\mathcal{PMF}(M)$ forms its boundary.  The closure  $\overline{\mathcal{T}(M)}$ is a closed ball on which the mapping class group $\MCG(M)$ acts continuously. Let $[f]$ be a pseudo\,-Anosov  isotopy class with invariant measured foliations $(\cF^s,\mu^s)$~and~$(\cF^u,\mu^u)$~and dilatation~$\lambda>1$. Let  $[\mathcal{F}^u,\mu^u]$  and $[\mathcal{F}^s,\mu^s]$ denote the  projective classes of its invariant foliations on $\mathcal{PMF}(M)$. The only fixed points of the induced action of $[f]$ in $\overline{\mathcal{T}}$ are $[\mathcal{F}^u,\mu^u]$~and $[\mathcal{F}^s,\mu^s]$ on $\mathcal{PMF}(M)$ and every other point on $\mathcal{PMF}(M)$ converges to $[\mathcal{F}^u,\mu^u]$ rapidly under the action of $[f]$.

		 The induced action of $[f]$ on $\mathcal{PMF}(M)$ is piecewise linear and is locally described by integer matrices. The matrix on any piece which contains $[\mathcal{F}^u,\mu^u]$ on its closure has an eigenvalue $\lambda>1$ since $[\mathcal{F}^u,\mu^u]$ is a fixed point on $\mathcal{PMF}(M)$. Therefore, in order to compute the dilatation of $[f]$ one should compute the action of $[f]$ on $\mathcal{PMF}(M)$ and find a matrix with an eigenvalue $\lambda>1$ with associated eigenvector contained in the relevant piece. Then the eigenvector corresponds to $[\mathcal{F}^u,\mu^u]$ and $\lambda$ gives the dilatation. In \cite{paper1} we realized this idea on $D_n$, coordinatizing the space of  measured foliations $\mathcal{MF}_n$  on $D_n$ using the Dynnikov coordinates and describing the action of $\MCG(D_n)$ on $\mathcal{MF}_n$ in terms of Dynnikov coordinates using the \emph{update rules} \cite{D02}.

		 The next section describes the Dynnikov coordinate system which puts global coordinates on $\mathcal{MF}_n$ \cite{D02, DDRW02, or08, paper1, paper2}.

\subsection{Dynnikov coordinates}\label{dynnikovsection}
Take a standard model of $D_n$ as depicted in Figure \ref{arcs}. Let $\cA_n$ be the set of arcs in~$D_n$~which have each endpoint either on the boundary or at a puncture. Consider the arcs $\alpha_i \in \cA_n$ ($1\le i\le 2n-4$) and $\beta_i\in \cA_n$ ($1\leq i\leq n-1$) as depicted in Figure~\ref{arcs}. Let $(\mathcal{F},\mu)\in\mathcal{MF}_n$. 

The {\em Dynnikov coordinate function} $\rho:\MF_n\to\R^{2n-4}\setminus\{0\}$  defined by

 %On the other hand,Denote by $\mu([\gamma]\in \cA_n$ the geometric intersection number of $\mathcal{L}\in \mathcal{L}_n$ with an arc $\gamma\in \cA_n $. 

\begin{eqnarray*}
\rho(\mathcal{F},\mu) = (a,b)=(a_1,\ldots,a_{n-2},\,b_1,\ldots,b_{n-2}),
\end{eqnarray*}
where for $1\le i\le n-2$

\begin{eqnarray*}
a_i=\frac{\mu([\alpha_{2i}])-\mu([\alpha_{2i-1}])}{2}\qquad
\text{and} \qquad b_i=\frac{\mu([\beta_i])-\mu([\beta_{i+1}])}{2} \label{dynnikov}
\end{eqnarray*}

\noindent is a homeomorphism \cite{D02, paper1, paper2}. 

Let $\mathcal{L}_n$ denote the set of integral laminations (disjoint unions of finitely many essential simple closed curves) on $D_n$. The Dynnikov coordinate function restricts to a bijection $\rho:\cL_n\to\Z^{2n-4}\setminus\{0\}$  \cite{paper1,paper2}. Figure \ref{arcs2} depicts the Dynnikov coordinates of an integral lamination $\mathcal{L}\in \mathcal{L}_5$.

\begin{figure}[h!]
\begin{center}
\psfrag{a1}[tl]{$\scriptstyle{\alpha_1}$} 
\psfrag{i+1}[tl]{$\scriptstyle{i}$} 
\psfrag{i}[tl]{$\scriptstyle{i-1}$} 
\psfrag{n}[tl]{$\scriptstyle{n}$} 
\psfrag{1}[tl]{$\scriptstyle{\alpha_1}$} 
\psfrag{b1}[tl]{$\scriptstyle{\textcolor{red}{\beta_1}}$} 
\psfrag{b2}[tl]{$\scriptstyle{\textcolor{red}{\beta_i}}$} 
\psfrag{bi}[tl]{$\scriptstyle{\textcolor{red}{\beta_{i-1}}}$}
\psfrag{b6}[tl]{$\scriptstyle{\textcolor{red}{\beta_{n-1}}}$} 
\psfrag{2i}[tl]{$\scriptstyle{\alpha_{2i}}$}
\psfrag{2}[tl]{$\scriptstyle{\alpha_2}$} 
\psfrag{2i-5}[tl]{$\scriptstyle{\alpha_{2i-5}}$}
\psfrag{2i-2}[tl]{$\scriptstyle{\alpha_{2i-2}}$}
\psfrag{2i-3}[tl]{$\scriptstyle{\alpha_{2i-3}}$}
\psfrag{2i-1}[tl]{$\scriptstyle{\alpha_{2i-1}}$}
\psfrag{2i-4}[tl]{$\scriptstyle{\alpha_{2i-4}}$}
\psfrag{n-4}[tl]{$\scriptstyle{\alpha_{2n-4}}$} 
\psfrag{n-5}[tl]{$\scriptstyle{\alpha_{2n-5}}$} 
\includegraphics[width=0.8\textwidth]{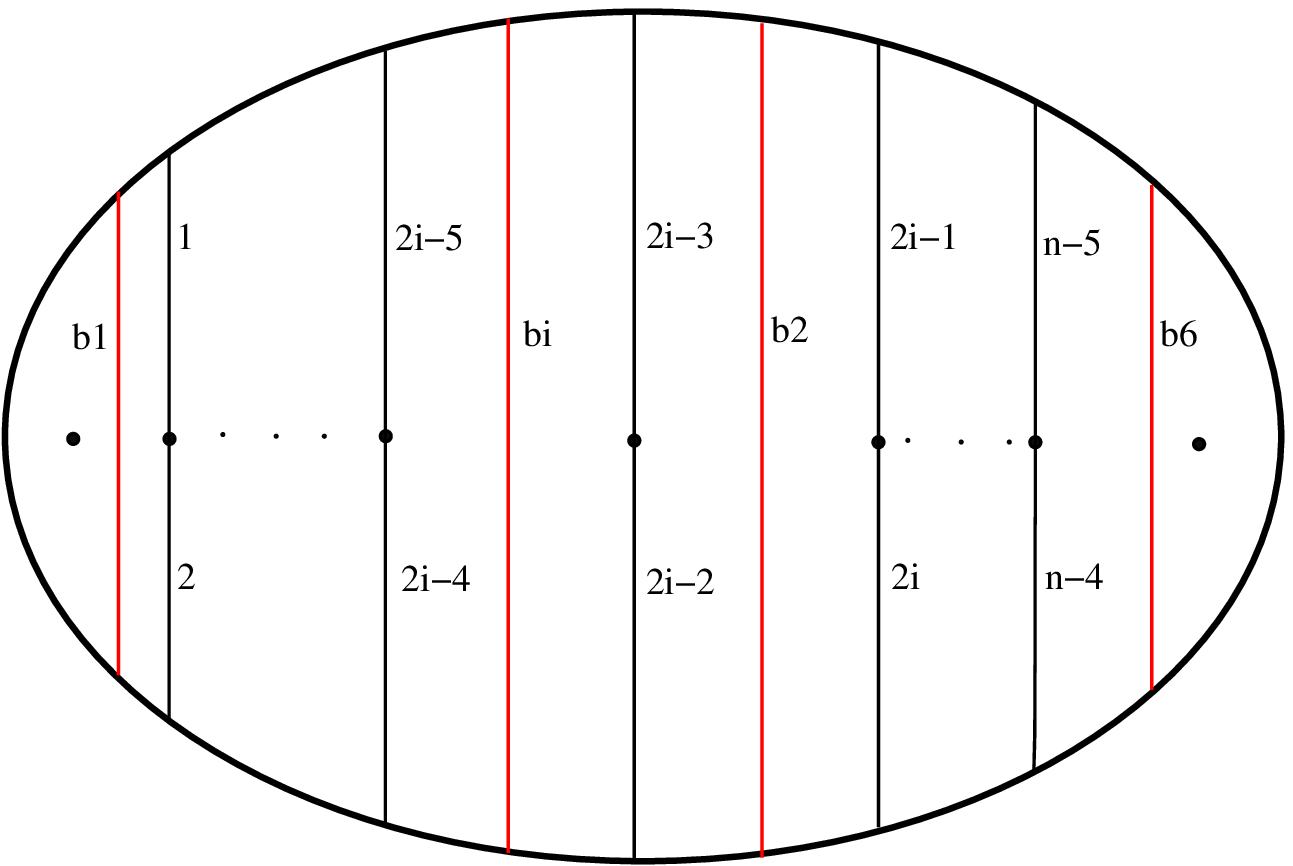}
%\resizebox{0.5\textwidth, angle=-90}{!}{\includegraphics{famcurves}}
\caption{The arcs $\alpha_i$ and $\beta_i$} \label{arcs}
%\end{center}
%\begin{center}
\psfrag{b1}[tl]{$\scriptstyle{6}$} 
\psfrag{bi}[tl]{$\scriptstyle{4}$} 
\psfrag{b2}[tl]{$\scriptstyle{2}$}
\psfrag{b6}[tl]{$\scriptstyle{2}$} 
\psfrag{2i}[tl]{$\scriptstyle{0}$}
\psfrag{2i-5}[tl]{$\scriptstyle{4}$}
\psfrag{2i-2}[tl]{$\scriptstyle{3}$}
\psfrag{2i-3}[tl]{$\scriptstyle{1}$}
\psfrag{2i-1}[tl]{$\scriptstyle{2}$}
\psfrag{2i-4}[tl]{$\scriptstyle{2}$}
\includegraphics[width=0.8\textwidth]{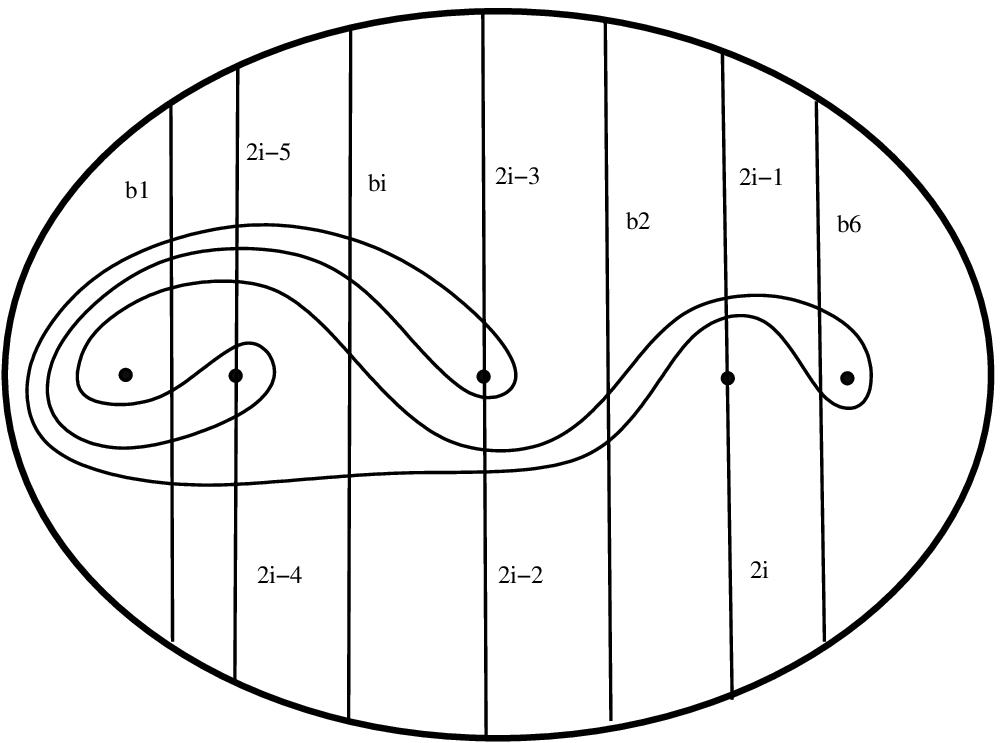}
%\resizebox{0.5\textwidth, angle=-90}{!}{\includegraphics{famcurves}}
\caption{$\rho(\mathcal{L})=(-1,1,-1,1,1,0)$.} \label{arcs2}
\end{center}
\end{figure}

Let $\cS_n=\R^{2n-4}\setminus \left\{0\right\}$ denote the space of Dynnikov coordinates. Projectivizing $\mathcal{S}_n$ we get a homeomorphism between $\mathcal{PMF}_n$ and $\mathcal{PS}_n$. The next section gives the \emph{update rules} which describe the action of $B_n$ on $\mathcal{S}_n$ and hence on $\mathcal{PS}_n$, and introduces Dynnikov matrices with illustrative examples.

%
%\begin{thm}\cite{paper1}[Inversion of Dynnikov coordinates]
%\label{lem:dynninvert}
%Let $(a,b) \in \cS_n$. Then $(a,b)$ is the Dynnikov coordinate of exactly one element $(\mathcal{F},\mu)$ of $\mathcal{MF}_n$, which has 
%\begin{align*}
%\beta_i&=2\max_{1\leq k\leq n-2}\left[|a_k|+\max(b_k,0)+\sum_{j=1}^{k-1}b_j\right]-2\sum_{j=1}^{i-1}b_j 
%\label{opopp}
%\end{align*}
%\begin{align*}
%\alpha_i&= \left\{ \begin{array}{ll}
        % (-1)^i a_{\lceil i/2 \rceil}+\frac{\beta_{\lceil i/2\rceil}}{2}& \mbox{if $b_{\lceil i/2\rceil} \geq 0$};\\
       % (-1)^i a_{\lceil i/2\rceil}+\frac{\beta_{1+\lceil i/2\rceil}}{2}& \mbox{if $b_{\lceil i/2\rceil} \leq 0$}   
         %\end{array} \right.
%\label{op}
%\end{align*}

%\noindent where $\lceil x \rceil$ denotes the smallest integer which is not less than $x$. 

%\end{thm}

\subsection{Update Rules and Dynnikov matrices}\label{updaterules}

Since $\MCG(D_n)$ is canonically isomorphic to Artin's braid group $B_n$ \cite{emil1,emil2}, the isotopy classes in $\MCG(D_n)$ are represented by  sequences of Artin's braid generators. The action of Artin's braid generators $\sigma_i$, $\sigma^{-1}_i,$ $(1 \leq i\leq n-1)$ on $\mathcal{MF}_n$ in terms of Dynnikov coordinates is described by the \textit{update rules} \cite{D02, paper1}.  Therefore, using the update rules  one can compute
 $\beta:\cS_n\to\cS_n$ given by,
$$\beta(a,b)=\rho\circ\beta\circ\rho\I(a,b)$$ for each $\beta\in B_n$.

\begin{figure}[h!]
\begin{center}
\psfrag{x}[tl]{$\scriptstyle{\sigma^{-1}_3\sigma_2\sigma^{-1}_1}$} 
\includegraphics[width=0.9\textwidth]{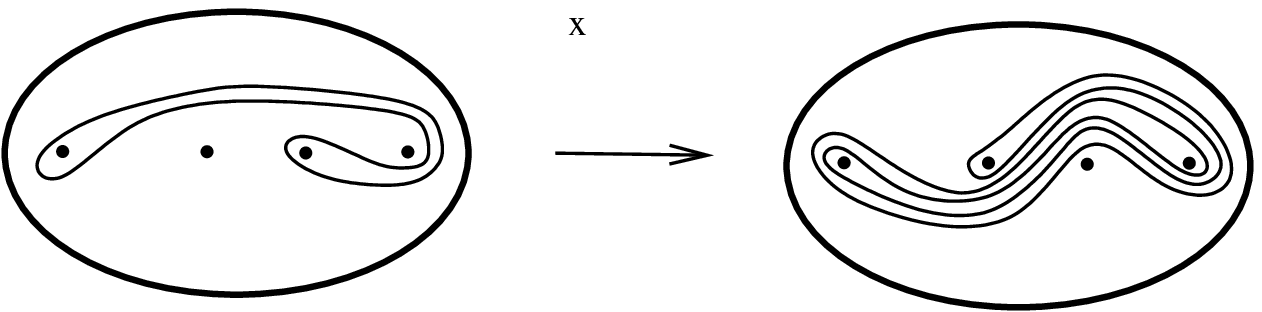}
%\resizebox{0.5\textwidth, angle=-90}{!}{\includegraphics{famcurves}}
\caption{$\rho(\mathcal{L})=(-1,-1,0,-1)$ and $\rho(\sigma^{-1}_3\sigma_2\sigma^{-1}_1(\mathcal{L}))=(2,-3,-1,0)$}\label{updateresim2}
\end{center}\end{figure}

		For computational and notational convenience, we will work in the {\em max-plus semiring} $(\R,\oplus,\otimes)$ \cite{paper1}. It will be convenient to use normal additive and multiplicative notation, and to indicate that these are to be interpreted in the max-plus sense by enclosing the formulae in square brackets. Therefore, $\left[a+b\right] = \max(a,b), \left[ab\right] = a+b, \left[a/b\right] = a-b, \left[1\right]= 0$.

For example, the formula 
$a_i'=\tropical{\frac{a_{i-1}a_ib_i}{a_{i-1}(1+b_i)+a_i}}$
will be just another way of writing
$a_i'=a_{i-1}+a_i+b_i-\max(a_{i-1}+\max(0,b_i),a_i).$
 Note that  both addition and multiplication are commutative and multiplication is distributive over addition:

\begin{align*}
\left[a+b\right]&=\max(a,b)=\max(b,a)=\left[b+a\right],\\
\left[ab\right]&=a+b=b+a=\left[ba\right],\\
\left[a(b+c)\right]&=a+\max(b,c)=\max(a+b,a+c)=\left[ab+ac\right].
\end{align*}
Let~$(a,b)\in\cS_n$ and \mbox{$1\le i\le n-1$}, and write 
$\sigma_i(a,b)=(a',b')$, $\sigma_i\I(a,b)=(a'',b'')$. 

\begin{thm}[see \cite{D02} and \cite{paper1}]\label{lem:update}

\leavevmode
\vspace{3mm}

\begin{itemize}
\item if $i=1$ then
\begin{align*}
a_1' &= \tropical{ \frac{a_1b_1}{a_1+1+b_1} }, &
b_1' &= \tropical{ \frac{1+b_1}{a_1} }\\*
a_1'' &= \tropical{ \frac{1+a_1(1+b_1)}{b_1}  }, &
b_1'' &= \tropical{ a_1(1+b_1) }; 
\end{align*}

\item if $2\le i \le n-2$ then
\begin{align*}
a_{i-1}' &= \tropical{ a_{i-1}(1+b_{i-1})+a_ib_{i-1}  }, &
b_{i-1}' &= \tropical{ \frac{a_ib_{i-1}b_i}{a_{i-1}(1+b_{i-1})(1+b_i)
+ a_ib_{i-1}}  }\\*
a_i' &= \tropical{ \frac{a_{i-1}a_ib_i}{a_{i-1}(1+b_i)+a_i}  }, &
b_i' &= \tropical{ \frac{a_{i-1}(1+b_{i-1})(1+b_i) + a_ib_{i-1}}{a_i} };\\*
a_{i-1}'' &= \tropical{
  \frac{a_{i-1}a_i}{a_{i-1}b_{i-1}+a_i(1+b_{i-1})}  },  &
b_{i-1}'' &= 
\tropical{ \frac{a_{i-1}b_{i-1}b_i}{a_{i-1}b_{i-1}+a_i(1+b_{i-1})(1+b_i)} }, \\*
a_i'' &= \tropical{ \frac{a_{i-1}+a_i(1+b_i)}{b_i}  }, &
b_i'' &= \tropical{ \frac{a_{i-1}b_{i-1}+a_i(1+b_{i-1})(1+b_i)}{a_{i-1}} };  
\end{align*}
\item if $i=n-1$ then
\begin{align*}
a_{n-2}' &= \tropical{ a_{n-2}(1+b_{n-2})+b_{n-2} }, &
b_{n-2}' &= \tropical{ \frac{b_{n-2}}{a_{n-2}(1+b_{n-2})} }\\*
a_{n-2}'' &= \tropical{\frac{a_{n-2}}{a_{n-2}b_{n-2}+1+b_{n-2}} }, &
b_{n-2}'' &= \tropical{ \frac{a_{n-2}b_{n-2}}{1+b_{n-2}} }.
\end{align*}
\end{itemize}
In all other cases $a_j'=a_j''=a_j$, $b_j'=b_j''=b_j$.

\end{thm}
\vspace{2mm}

\begin{example}\label{illustrative}
Let $\mathcal{PMF}_3\cong S^1$ be the space of projective measured foliations on $D_3$.  Using Theorem~\ref{lem:update}, that is using the update rules given there,  one can explicitly compute the $2\times 2$ integer matrices which describe the piecewise linear action of $\beta=\sigma_1\sigma^{-1}_2$ on $\mathcal{PMF}_3$.

\begin{figure}[h!]
\begin{center} 
\psfrag{2}[tl]{$\tiny{\left[ \begin {array}{cc} 1&-1\\\noalign{\medskip}1&0
\end {array} \right]}$}
\psfrag{A}[tl]{$\tiny{A}$}
\psfrag{B}[tl]{$\tiny{B}$}
\psfrag{a}[tl]{$\tiny{a}$}
\psfrag{b}[tl]{$\tiny{b}$}
\psfrag{C}[tl]{$\tiny{C}$}
\psfrag{D}[tl]{$\tiny{D}$}
\psfrag{E}[tl]{$\tiny{E}$}
\psfrag{l1}[tl]{$\tiny{\ell_1}$}
\psfrag{l2}[tl]{$\tiny{\ell_2}$}
\psfrag{F}[tl]{$\tiny{F}$}
\psfrag{3}[tl]{$\tiny{\left[ \begin {array}{cc} 1&-1\\\noalign{\medskip}-1&2
\end {array} \right]}$}
\psfrag{4}[tl]{$\tiny{\left[ \begin {array}{cc} 0&1\\\noalign{\medskip}-1&2
\end {array} \right]}$}
\psfrag{1}[tl]{$\tiny{\left[ \begin {array}{cc} 2&-1\\\noalign{\medskip}1&0
\end {array} \right]}$}
\psfrag{5}[tl]{$\tiny{\left[ \begin {array}{cc} 0&1\\\noalign{\medskip}-1&1
\end {array} \right]}$}
\psfrag{6}[tl]{$\tiny{\left[ \begin {array}{cc} 2&1\\\noalign{\medskip}1&1
\end {array} \right]}$}
\psfrag{a=b}[tl]{${a=b}$}
\psfrag{a=2b}[tl]{${a=2b}$}
\includegraphics[width=0.85\textwidth]{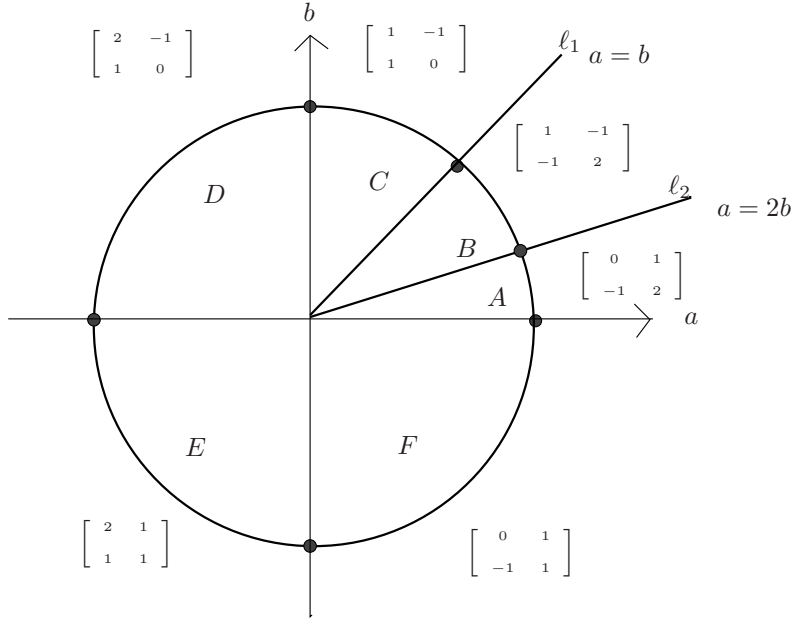}
\caption{The action of $\sigma_1\sigma^{-1}_2$ on $\mathcal{PS}_3$}
\label{fig:teichmuller1}
\end{center}
\end{figure}

We remark that  this example concretely illustrates the action of a pseudo\,-Anosov  braid (which is the simplest possible) on the whole space $\mathcal{PMF}_3$.  See Figure \ref{fig:teichmuller1}.

Let $\ell_1$ and $\ell_2$ denote the lines $a=b$ and $a=2b$ in the first quadrant. Let $+a$~($-a$) and $+b~(-b)$ denote the positive (negative) $a$-axis and $b$-axis respectively. Write $x\rightarrow y$ if $x$ is sent onto $y$ by the action of $\sigma_1\sigma^{-1}_2$. Then we have $\ell_1\rightarrow+b\rightarrow-a$, $\ell_2\rightarrow+a\rightarrow-b$ and then $-a$ and $-b$ are sent into the interior of $E$. Hence $A\rightarrow F \rightarrow E$ and $C \rightarrow D\rightarrow E$. Therefore, we can see that $p^u=[a^u,b^u]$ lies in region $E$ and $p^s=[a^s,b^s]$ lies in region $B$.

Indeed the matrix
$$D=\left[ \begin {array}{cc} 2&1\\\noalign{\medskip}1&1
\end {array} \right]
$$
has an eigenvalue $\lambda=\frac{3+\sqrt{5}}{2}$, and the corresponding eigenvector $p^u=-(\frac{1+\sqrt {5}}{2},1)$ belongs to $E$ since $a^u\leq 0$ and $b^u\leq 0$. Hence $p^u$ is a fixed point for  $\sigma_1\sigma_2^{-1}$ on $\mathcal{PS}_3$.
Hence, $p^u=[a^u,b^u]$ corresponds to the invariant unstable foliation $[\mathcal{F}^u,\mu^u]$.  Similarly, the matrix  
$$
\left[ \begin {array}{cc} 1&-1\\\noalign{\medskip}-1&2
\end {array} \right]
$$
has an eigenvalue $1/\lambda$ and the associated eigenvector $p^s$ belongs to the region $a\geq 0,~b\geq 0,~ b\leq a\leq 2b$. Hence $p^s$ is a fixed point and corresponds to the invariant stable foliation $[\mathcal{F}^s,\mu^s]$.

\end{example}

The collection of linear equations (not necessarily independent) in various maxima in update rules give each region on $\mathcal{PS}_n$ the structure of a polyhedron. 
One can see this in Figure \ref{fig:teichmuller1} by observing that each region is a solution set for a system of linear inequalities induced by these equations.
Let $p^u=(a^u,b^u)$  and $p^s=(a^s,b^s)$ denote the Dynnikov coordinates of $(\mathcal{F}^u,\mu^u)$ and $(\mathcal{F}^s,\mu^s)$ respectively. Next we define the \emph{Dynnikov matrices} which describe the action of $\beta$ near~$p^u$.

\begin{mydef}\label{dynndef}

Let $\beta\in B_n$ be a pseudo\,-Anosov  braid with invariant unstable measured foliation $(\mathcal{F}^u,\mu^u)$ given by the Dynnikov coordinates $p^u=(a^u,b^u)$.  The action of $\beta$ on $\mathcal{S}_n$ is piecewise linear and each closed piece $\mathcal{R}_i\subset \mathcal{S}_n$ containing $(a^u,b^u)$ is called a \emph{Dynnikov region}. 
Then a \emph{Dynnikov matrix} $D_i:R_i\to\mathcal{S}_n$, ($1\leq i \leq k$) is a $(2n-4)\times(2n-4)$ integer matrix which describes the behaviour of the braid on a Dynnikov region $\mathcal{R}_i$.  That is, 
\begin{align*}
\rho(\beta(\mathcal{F},\mu))=D_i(a,b) ~~\text{for}~~ (a,b) \in R_i.
\end{align*}
\end{mydef}

In contrast with the example above there can be more than one Dynnikov region for a given pseudo\,-Anosov  braid $\beta$. This happens when $(a^u,b^u)$ is on the boundary of several regions on $\mathcal{S}_n$.
\begin{example}
Let $\beta=\sigma_1\sigma_2\sigma_3\sigma^{-1}_4\in B_5$. Using Theorem~\ref{updaterules}  we compute that the action of $\sigma_1\sigma_2\sigma_3\sigma^{-1}_4$ on $\mathcal{S}_5$ in a region where $a_i\leq 0$ and $b_i\leq 0$ for all $i$ is given by two matrices $D_1$ (when 
$b_1\leq a_2-a_1$) and $D_2$ (when $b_1\geq a_2-a_1$) where $D_1$ and $D_2$ are given as follows.

\begin{equation*}
\tiny D_1=\begin{pmatrix}-1&1&0&0&0&0\\ \noalign{\medskip}0&0&0&
1&1&0\\ \noalign{\medskip}0&0&2&-1&-1&1\\ \noalign{\medskip}0&0&0&0&1&0
\\ \noalign{\medskip}-1&0&1&-1&-1&1\\ \noalign{\medskip}0&0&1&0&0&1
\end{pmatrix},~~D_2=\begin{pmatrix}0&0&0&1&0&0\\ \noalign{\medskip}0&0&0&1
&1&0\\ \noalign{\medskip}0&0&2&-1&-1&1\\ \noalign{\medskip}-1&1&0&-1&1
&0\\ \noalign{\medskip}0&-1&1&0&-1&1\\ \noalign{\medskip}0&0&1&0&0&1
\end{pmatrix}
\end{equation*} 

\normalsize
\noindent Both of these matrices have eigenvalue $\frac{3+\sqrt{5}+\sqrt{6\sqrt{5}-2}}{4}$, with the corresponding eigenvector $p^u$ having all negative entries and satisfying the equality $a_2= a_1+b_1$. Therefore, both $D_1$ and $D_2$ are Dynnikov matrices. 
		
		\end{example}

\subsection{Main results}\label{Background}
Knowing from \cite{los, paper1} that both Dynnikov matrices and train track transition matrices record the action of a pseudo-Anosov braid on the same space $\mathcal{PMF}_n$, we would like to know exactly what the new action tells us. Note that the dimensions of a Dynnikov and train track transition matrix for the same braid are in general different.

		 Our aim in this paper is to show that  Dynnikov matrices and train track transition matrices of a given pseudo-Anosov isotopy class are isospectral up to roots of unity and zeros under some particular conditions.  It turns out that the isospectrality of these matrices depends on the singularity structure of the invariant foliations and how the prongs (a leaf starting at a singularity) of the singularities are permuted under the action of the isotopy class. We also note that this result is relatively straightforward when the unstable measured foliation lies in a single Dynnikov region: the difficulty arises in dealing with the case when it lies on the boundary of several regions, particularly when the action of the isotopy class permutes the regions, in which case the spectrum of the various Dynnikov matrices does not have an obvious interpretation (see Remark \ref{sonornek}).

		Our first main result is for the case where $(\mathcal{F}^u,\mu^u)$ has only unpunctured $3$-pronged and punctured $1$-pronged singularities. The notion of a ``regular'' train track in the following theorems will be clarified in Section \ref{subsection21}.
\begin{thm}\label{firstthm}
Let $\beta\in B_n$ be a pseudo\,-Anosov braid with unstable invariant foliation $(\mathcal{F}^u,\mu^u)$ and dilatation $\lambda>1$. Let  $\tau$ be a regular invariant train track with associated transition matrix $T$. If $(\mathcal{F}^u,\mu^u)$ has only unpunctured $3$-pronged and punctured $1$-pronged singularities, then $\beta$ has a unique Dynnikov matrix $D$, and $D$ and $T$ are isospectral.
\end{thm}
 A similar result (Corollary \ref{anycomplete}) follows for general train tracks from the next two results. The proof of Lemma \ref{eski} is given in Section \ref{section2part2}. The author noticed that a similar proof was given in \cite{b10}. 

\begin{thm}[Rykken \cite{R99}]\label{rykken}
Let $f: M\to M$ be a pseudo\,-Anosov  homeomorphism on an orientable surface $M$ of genus $g$ with oriented unstable manifolds. Let $T$ be a train track transition matrix for $f$. If $f$ preserves the orientation of unstable manifolds, then the eigenvalues of $f_{1*}:H_1(M;\mathbb{R})\to H_1(M;\mathbb{R})$ are the same as those of $T$, including multiplicity, up to  roots of unity and zeros.
\end{thm}

\begin{lem}\label{eski}
Let $\beta$ be a pseudo\,-Anosov  isotopy class on $D_n$ with invariant train track $\tau$ and  associated transition matrix $T$. Let $\tilde{f}$ be the lift of $f$ to the orientation double cover $M$. Let $\tilde{\tau}$ and $\tilde{T}$ be the lifted invariant train track and transition matrix associated to $[\tilde{f}]$.
Then $T$ and $\tilde{T}$ are isospectral up to roots of unity. 
\end{lem}

Our following results are for the case when $(\mathcal{F}^u,\mu^u)$ has singularities other than unpunctured 3-pronged and punctured 1-
pronged singularities. Lemma \ref{chartrelation} will play a key role in proving our following results.

\begin{thm}\label{thm1}
Let $\beta\in B_n$ be a pseudo\,-Anosov braid with unstable invariant measured foliation $(\mathcal{F}^u,\mu^u)$ and dilatation $\lambda>1$. Let $\tau$  be a regular invariant train track of $\beta$ with associated transition matrix $T$. If $\beta$ fixes the prongs at all singularities other than unpunctured $3$-pronged and punctured $1$-pronged singularities, then any Dynnikov matrix $D_i$ is isospectral to $T$  up to some eigenvalues~$1$.
\end{thm}

Again a similar result (Corollary \ref{anytraintrack2}) follows for general train tracks from Theorem \ref{rykken} and Lemma \ref{eski}. The next theorem shows that if $D_n-\tau$ has only odd-gons all of the Dynnikov matrices are equal and hence there is only one Dynnikov region in the fixed-pronged case.

\begin{thm}\label{uniquematrix}
Let $\beta\in B_n$ be a pseudo\,-Anosov braid with unstable invariant measured foliation $(\mathcal{F}^u,\mu^u)$ and dilatation $\lambda>1$. Let $\tau$  be a regular invariant train track of $\beta$ with associated transition matrix $T$. If all components of $D_n-\tau$ are odd-gons and $\beta$ fixes the prongs at all singularities other than unpunctured $3$-pronged and punctured $1$-pronged singularities, then there is a unique Dynnikov region. 
\end{thm}

We note that the Dynnikov matrices and train track transition matrices throughout this paper  were computed using Dynnikov and train track programs implemented by Toby Hall both of which can be found at \cite{toby}.

In the next section we shall review some necessary background from \cite{bh95, mosher1, penner} that are necessary to prove our results in Section \ref{mainresults}.
\section{Proof of theorems \ref{firstthm}, \ref{thm1} and \ref{uniquematrix}}\label{section2}

\subsection{Train track coordinates and transition matrices}\label{subsection21}

 A \emph{train track} $\tau$ on $D_n$ is a one dimensional CW complex made up of vertices (\emph{switches}) and edges (\emph{branches}) smoothly embedded on $D_n$ such that at each switch there is a unique tangent vector, and every component of $D_n-\tau$ is either a once-punctured $p$-gon with $p\geq 1$ or an unpunctured $k$-gon with $k\geq 3$ (where the boundary of $D_n$ is regarded as a puncture). A train track $\tau$ is called \emph{complete} if each component of $D_n-\tau$ is either a trigon or a once punctured monogon. A transverse measure on $\tau$ is a function which assigns a  measure to each branch of~$\tau$ such that these measures satisfy the \emph{switch conditions} at each switch of $\tau$. That is, for each switch $v$ of $\tau$ $$\displaystyle\sum_{\text{incoming~ branches ~at}~ v}\mu(e)=\displaystyle\sum_{\text{outgoing ~branches~ at}~ v}\mu(e)$$
 Train tracks equipped with a transverse measure  are called \emph{measured train tracks}: they provide another way to coordinatize measured foliations and integral laminations. We denote by $\mathcal{W}(\tau)$  and $\mathcal{W}^+(\tau)$ the space of transverse measures and non-negative transverse measures associated to $\tau$.

\subsection*{Constructing measured foliations from train tracks} Given a measured train track $\tau$ define a function $\phi_\tau:\mathcal{W}^+(\tau)\to \mathcal{MF}_n$ as follows: Replace each branch $e_i$ of $\tau$ which has non-zero measure with a Euclidean rectangle $R_i$ of length $1$ and height $\mu(e_i)$ and endow each $R_i$ with a ``horizontal" measured foliation where the transverse measure is induced from the Euclidean metrics on the rectangles. At each switch glue the vertical sides of the rectangles and denote this union of glued rectangles $\mathcal{R}^*$. Since $\tau$ satisfies the switch condition at each switch there is a unique measure preserving way to glue together the horizontal leaves, hence there is  a well defined transverse measure on $\mathcal{R}^*$. A \emph{pre-foliation} $\mathcal{F}^*$ is the collection of leaves on $\mathcal{R}^*$. Collapsing each component of $D_n-\mathcal{F}^*$ which doesn't contain any branch of zero measure onto a spine yields a  measured foliation $\phi_\tau(\mu)=(\mathcal{F},\mu)$ \cite{bh95,penner}. We say that $(\mathcal{F},\mu)\in \mathcal{MF}_n$ is \emph{carried} by $\tau$ if it arises from some transverse measure $\mu$ on $\tau$ in this way. We write $\mathcal{MF}(\tau)=\phi_{\tau}(\mathcal{W}^+(\tau))$ for the set of measured foliations carried by~$\tau$ and $\mathcal{PMF}(\tau)$ for the corresponding projective space. 
\begin{remark}
Note that if $(\mathcal{F},\mu)$ has only $1$-pronged singularities at punctures and $3$-pronged singularities elsewhere it is carried by a complete train track. It is carried by a non-complete train track otherwise.
\end{remark}

Since $\mathcal{MF}_n$ and $\mathcal{PMF}_n$ are homeomorphic to $\mathbb{R}^{2n-4}\setminus\{0\}$ and $\mathbb{S}^{2n-5}$ respectively, $\mathcal{MF}(\tau)$ and $\mathcal{PMF}(\tau)$ have the subspace topology. Furthermore, the functions $\phi_{\tau}:\mathcal{W}^{+}(\tau)\to\mathcal{MF}(\tau) $ and $\hat{\phi}_{\tau}:\mathcal{PW}^+(\tau)\to\mathcal{PMF}(\tau)$ are homeomorphisms where $\mathcal{W}^+(\tau)$ has the subspace and $\mathcal{PW}^+(\tau)$ has the quotient topology. Let  $\rank({\tau})$ denote the dimension of $\mathcal{W}(\tau)$. The switch conditions on $\tau$ are linearly independent and hence $\rank(\tau)=k-s$ where $k$ is the number of branches and $s$ the number of switches of $\tau$. Therefore, $\mathcal{W}(\tau) \cong\mathbb{R}^{k-s}\setminus \{0\}$ and $\tau$ is complete if and only if $\rank(\tau)=2n-4$. That is,  $\tau$ is complete if and only if $\rank(\tau)$ is the same as the dimension of ${\mathcal{MF}}_n$. The complete train tracks on $D_n$ give an atlas for the piecewise integral linear structure of $\mathcal{MF}_n$ and $\mathcal{PMF}_n$. That is, the transition functions between charts are piecewise linear with integer coefficients \cite{mosher1, penner}.

\begin{mydefs}\label{carryingmap}
Endowing a regular neighborhood $N_{\tau}$ of $\tau$ with fibres of the retraction $r:N_{\tau}\searrow\tau$, we obtain a \emph{fibred neighbourhood} $N_\tau$ of $\tau$. Let $\tau$ and $\tau'$ be two train tracks on $D_n$. We say that $\tau$ is \emph{carried} by $\tau'$ and write $\tau<\tau'$ if there is a homeomorphism $\psi:D_n\to D_n$ isotopic to the identity such that
\begin{itemize}
\item $\psi(\tau)\subseteq N_{\tau'}$,
\item Each branch of $\psi(\tau)$ is transverse to the fibers in $N_{\tau'}$,
\item for each branch $e_i$ of $\tau$ the end points of $\psi(e_i)$ are contained in singular leaves of $N_{\tau'}$.
\end{itemize}
Let $\{e_i\}_{1\leq i\leq k}$  and $\{f_i\}_{1\leq i\leq k'}$ be the oriented branches of $\tau$ and $\tau'$ respectively. Let $r':N_{\tau'}\to\tau'$ be the retraction. For each $1\leq i\leq k$, $r'(\psi(e_i))$ is an edge path in $\tau'$: $r'(\psi(e_i))=f^{\epsilon_1}_{i_1}f^{\epsilon_2}_{i_2}\dots f^{\epsilon_s}_{i_s}$,  ${\epsilon_j}=\pm 1$. The \emph{incidence matrix} associated to $\tau$ and $\tau'$ is the $k'\times k$ matrix $G:\mathcal{W}(\tau)\to\mathcal{W}(\tau')$ whose $ij^{\text{th}}$ entry $G_{ij}$ is given by  the number of occurences of  $f^{\pm 1}_i$ in $r'(\psi(e_j))$.
%$\tau$ is isotopic to a train track $\sigma\subseteq N_{\tau'}$ which is transverse to the fibers in $N_{\tau'}$. 
\end{mydefs}

\begin{lem}[\cite{mosher1}]\label{commutingdiagram}
Let $\tau<\tau'$. Then $\mathcal{MF}(\tau)\subset\mathcal{MF}(\tau')$ and the following diagram commutes:

$$\begin{array}[c]{ccc}
\mathcal{W}^+(\tau)&\stackrel{G}{\longrightarrow}&\mathcal{W}^+(\tau')\\
\downarrow\scriptstyle{\phi_{\tau}}&&\downarrow\scriptstyle{\phi_{\tau'}}\\
\mathcal{MF}(\tau)&{\hookrightarrow}&\mathcal{MF}(\tau').
\end{array}$$
\end{lem}

\begin{mydef}
A train track $\tau$ is \emph{invariant} under $\beta\in B_n$, if $\beta(\tau)$ is carried by $\tau$. Let $\tau$ be an invariant train track of $\beta$ and $e_1,\dots,e_k$ be the oriented branches of $\tau$. Then, for each $1\leq i \leq k$, $r(\psi(\beta(e_i)))$ is of the form $r(\psi(\beta(e_i)))=e^{\epsilon_1}_{i_1}e^{\epsilon_2}_{i_2}\dots e^{\epsilon_k}_{i_k}$,  ${\epsilon_j}=\pm 1$. The \emph{transition matrix} $T$ associated to $\tau$ is the $k\times k$ incidence matrix $T:\mathcal{W}(\beta(\tau))\to \mathcal{W}(\tau)$ described as in Definitions \ref{carryingmap}.
\end{mydef}

\begin{thm}[\cite{bh95}]\label{bh1}
Every pseudo\,-Anosov  braid $\beta\in B_n$ has an invariant train track $\tau$. This train track $\tau$ can be chosen so that

\begin{itemize}
 \item The branches which bound interior $p$-gons (i.e. those which are disjoint from~$\partial D_n$) are permuted by $\beta$
 \item The transition matrix is of the form 
$$
T'= \left(\begin{array}{cc}
T & 0 \\
A& P \\
\end{array} \right)$$
where $P$ is a permutation matrix giving the action on the permuted branches and $T$ is the matrix that gives the action on the other branches.
\item For each $p$, there are the same number of unpunctured (resp. punctured) $p$-gons in $\tau$ as there are unpunctured (resp. punctured) $p$-pronged singularities in $(\mathcal{F}^u,\mu^u)$ (this includes the ``exterior" punctured $p$-gon and the singularity at infinity).
\end{itemize}
\end{thm}

We shall call a train track $\tau$ of the type in Theorem~\ref{bh1} a \emph{regular train track} \cite{bh95}. A branch of a regular train track is called \emph{infinitesimal} if it is permuted under the action of $\beta$ (that is, if it bounds an interior $p$-gon), it is called \emph{main} otherwise.

\begin{figure}[h!]
\begin{center}
\psfrag{1/2}[tl]{$\scriptstyle{\frac{1}{2}}$} 
\psfrag{-1/2}[tl]{$\scriptstyle{-\frac{1}{2}}$} 
\psfrag{0}[tl]{$\scriptstyle{0}$} 
\includegraphics[width=0.7\textwidth]{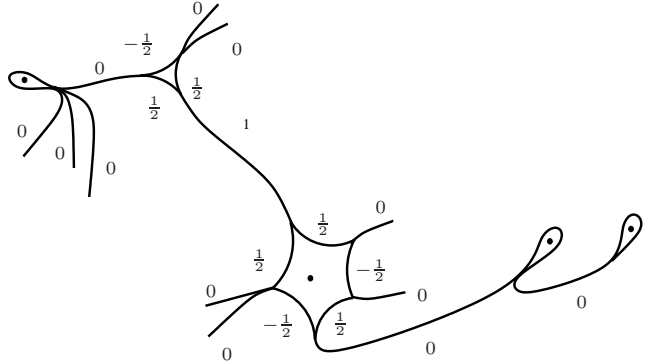}
\caption{Constructing a basis for $\mathcal{W}(\tau)$ when $D_n-\tau$ consists of odd-gons}\label{basis}
\end{center}
\end{figure}

\begin{lem}[\cite{bh95}]\label{perronfoliation}
Let $\beta\in B_n$ be a  pseudo\,-Anosov  braid with dilatation $\lambda$, unstable foliation $(\mathcal{F}^u,\mu^u)$ and invariant regular train track $\tau$ with associated transition matrix $T'$. 
The largest eigenvalue of $T'$ equals $\lambda$ and the entries of the unique associated column eigenvector $v^u$ (up to scale) are strictly positive. $v^u$ defines a transverse measure on $\tau$ which yields a pre-foliation $\mathcal{F}^*$ as described above whose prongs do not join and from which $(\mathcal{F}^u,\mu^u)$ is constructed. That is, $\phi_{\tau}(v^u)=(\mathcal{F}^u,\mu^u)$.
\end{lem}
See \cite{b10} for a proof Lemma \ref{inducedaction}. See also Figure \ref{basis}.  
\begin{lem}\label{inducedaction}
Let $\tau$ be an invariant train track for $\beta\in B_n$. If all components of $D_n-\tau$ are odd-gons (in particular, if $\tau$ is complete), then there is a basis for $\mathcal{W}(\tau)$ consisting of transverse measures $\mu$ such that $\mu(e)\neq 0$ for exactly one main branch $e$. 
\end{lem}

\begin{remark}
Lemma \ref{commutingdiagram} gives the following commutative diagram

$$\begin{array}[c]{ccc}
\mathcal{W}^+(\tau)&\stackrel{T'}{\longrightarrow}&\mathcal{W}^+(\tau)\\
\downarrow\scriptstyle{\phi_{\tau}}&&\downarrow\scriptstyle{\phi_{\tau}}\\
\mathcal{MF}(\tau)&\stackrel{\beta}{\longrightarrow}&\mathcal{MF}(\tau).
\end{array}$$
Hence the action on $\mathcal{W}(\tau)\subseteq \mathbb{R}^k$ is given by the $k\times k$ transition matrix 
$$
T'= \left(\begin{array}{cc}
T & 0 \\
A& P \\
\end{array} \right)$$
where $T$ is the $m\times m$ matrix which gives the action on the main branches of $\tau$
by Lemma \ref{bh1}. 
When all components of $D_n-\tau$ are odd-gons, a basis for $\mathcal{W}(\tau)\cong\mathbb{R}^m$ can be constructed as described in Lemma \ref{inducedaction} and hence the action on $\mathcal{W}(\tau)\cong\mathbb{R}^m$ 
is given by the  $m\times m$ transition matrix $T$.

We note that such a basis can not be taken if $D_n-\tau$ has an even-gon since the switch conditions are satisfied only when the alternating sum of the incoming measures on the switches of each even-gon is zero. However, there is still a basis consisting of weights on edges which may include infinitesimal ones, see \cite{b10}. 
\end{remark}

\subsection{Train track coordinates and Dynnikov coordinates}\label{section1part2}

In this section we show that, for any train track $\tau$ on $D_n$, the change of coordinate function $L:\mathcal{W}^+(\tau)\to \mathcal{S}_n$ between train track coordinates and Dynnikov coordinates is piecewise linear.
\begin{mydefs}\label{piecewiseclear}
Suppose $\tau$ is a train track with oriented branches $e_1,\dots,e_k$.  A \emph{train path} $p=e^{\epsilon_1}_{i_1}e^{\epsilon_2}_{i_2}\dots e^{\epsilon_m}_{i_m}$, ${\epsilon_j}=\pm 1$ is a smooth oriented edge-path in $\tau$. Given a train path $p$ on $\tau$ we define $\hat{p}:\mathcal{W}^+(\tau)\to \mathbb{R}_{\geq 0}$ as follows:  for each $\mu\in \mathcal{W}^+(\tau)$, $\hat{p}(\mu)$ is the total measure of leaves of $\phi_{\tau}(\mu)$ following  
the train path $p$.
\end{mydefs}

\begin{lem}\label{piecewisemeasure}
For each train path $p$ in $\tau$, the map $\hat{p}:\mathcal{W}^+(\tau)\subseteq \mathbb{R}_+^k \to \mathbb{R}_{\geq 0}$ is piecewise linear.
\end{lem}

\begin{proof}
Let $R_1,\dots R_k$ be the rectangles used in the construction of $\phi_{\tau}(\mu)$ as described above. Associate a copy of $R_{i_j}$ to each branch $e^{\epsilon_j}_{i_j}$ in $p$, and glue $R_{i_j}$ to $R_{i_{j-1}}$ and $R_{i_{j+1}}$ using the described identifications. Denote the identification space $K$. Then $\hat{p}(\mu)$ is the width of the largest rectangle which fits in $K$ with edges parallel to the edges of each $R_{i_j}$. This is clearly a piecewise linear function of the widths of the rectangles (observe that $\hat{p}(\mu)$ is the measure of leaves that pass along the shaded rectangle in Figure \ref{generalschema}).
\end{proof}

\begin{remark}\label{notconnected}
Note that Lemma \ref{piecewisemeasure} implies that  $\hat{p}:\mathcal{W}^+(\tau)\subseteq \mathbb{R}_+^k \to \mathbb{R}_{\geq 0}$ is linear in a neighbourhood of any measure for which the prongs of pre-foliation $\mathcal{F}^*$ are not connected.
\end{remark}

\begin{figure}
\centering
\psfrag{a}[tl]{$\scriptstyle{a}$}
\psfrag{x}[tl]{$\scriptstyle{x}$}
\psfrag{p}[tl]{$\scriptstyle{p}$}
\psfrag{e}[tl]{$\scriptstyle{e}$}
\psfrag{b}[tl]{$\scriptstyle{b}$}
\psfrag{c}[tl]{$\scriptstyle{c}$}
\psfrag{d}[tl]{$\scriptstyle{d}$}
\psfrag{h}[tl]{$\scriptstyle{h}$}
\psfrag{f}[tl]{$\scriptstyle{f}$}
\psfrag{g}[tl]{$\scriptstyle{g}$}
\includegraphics[width=0.6\textwidth]{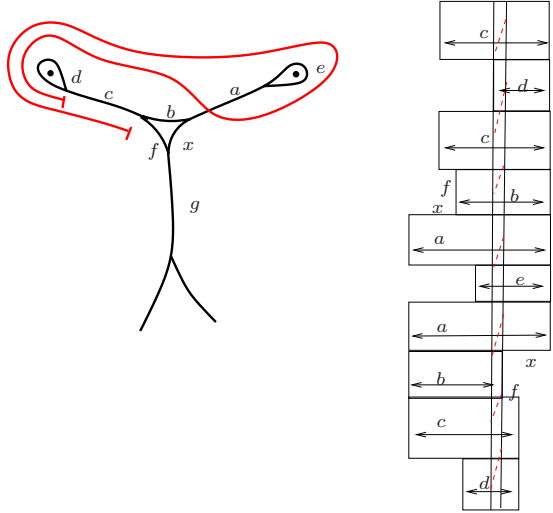}
\caption{The space $K$ for a train path $p$}\label{generalschema}
\end{figure}

\begin{mydefs}\label{tightminimal}
Let $\tau$ be a train track on $D_n$ and $\mathcal{A}_n\subset D_n$ be the set of Dynnikov arcs $\alpha_i$ $(1\leq i\leq 2n-4)$ and $\beta_i$  $(1\leq i\leq n-1)$. A \emph{standard embedding of $\tau$} in $D_n$ with respect to $\mathcal{A}_n$ satisfies the following:  
\begin{itemize}
\item each branch $e_i$ of $\tau$ is tight (that is $e_i$ doesn't bound any unpunctured disk with any Dynnikov arc)
\item the arcs $\alpha_i$ $(1\leq i\leq 2n-4)$ and $\beta_i$  $(1\leq i\leq n-1)$ do not pass through the switches of $\tau$.
\end{itemize}

We shall always take a standard  embedding of a given train track $\tau$ as described in Definitions \ref{tightminimal} throughout the text.

	 We say that a train path $p$ in $\tau$ is \emph{non-tight} with respect to a Dynnikov arc~$\gamma$ if some subarc of $p$ together with some subarc of $\gamma$ bounds a disk containing no punctures. For each Dynnikov arc $\gamma$ write $\Pi_{\gamma}$ for the set of all train paths which are non-tight with respect to $\gamma$. There is a partial order $\leq$ on $\Pi_{\gamma}$ defined as follows: $p_1\leq p_2$ if $p_1$ is a subpath of $p_2$. We define $\Pi'_{\gamma}\subseteq \Pi_{\gamma}$ as the subset of minimal train paths with respect to the relation $\leq$. Then any minimal non-tight train path $p \in \Pi'_{\gamma}$ is the concatenation  $p=\delta_1\delta_2\delta_3$ of three paths (not train paths) where $\delta_2$ is the subarc bounding a disk with some subarc of $\gamma$ and $\delta_1$ and $\delta_3$ are contained in single branches of $\tau$ (Figure \ref{snttornek}). 
\end{mydefs}

\begin{figure}
\centering
\psfrag{p1}[tl]{$\scriptstyle{p}$}
\psfrag{p2}[tl]{$\scriptstyle{p_2}$}
\psfrag{gamma}[tl]{$\scriptstyle{\gamma}$}
\psfrag{2}[tl]{$\scriptstyle{\tilde{p}_2}$}
\psfrag{3}[tl]{$\scriptstyle{e^{\epsilon_4}_{i_4}}$}
\psfrag{4}[tl]{$\scriptstyle{\delta_1}$}
\psfrag{6}[tl]{$\scriptstyle{\delta_3}$}
\psfrag{5}[tl]{$\scriptstyle{\delta_2}$}
\includegraphics[width=0.5\textwidth]{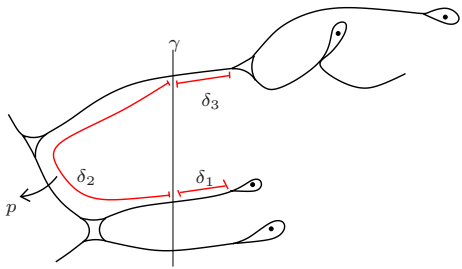}
\caption{A minimal non-tight train path}\label{snttornek}
\end{figure}

\begin{lem}\label{piecewiselin}
Let $\tau$ be a train track. Then, the change of coordinate function $L:\mathcal{W}^+(\tau)\to \mathcal{S}_n$ is piecewise linear.
\end{lem}
\begin{proof}
Given a Dynnikov arc $\gamma$ and the branches $\{e_1,\dots,e_k\}$ of $\tau$ (standardly embedded) let $n_i$ be the number of intersections of $e_i$ with $\gamma$.  To compute $\mu(\gamma)$ we need to subtract the measure on all independent train paths which form a loop with ${\gamma}$. The only condition is that a train path should not be a subpath of another (two train paths neither of which is a subpath of the other define disjoint packets of leaves except perhaps their boundary leaves).

 Thus the measure $\mu(\gamma)$ of $\gamma$ is given by  \begin{align*}\label{equationpiecewise}
 \mu(\gamma)=\sum^k_{i=1}n_i\mu(e_i)-2\sum_{p\in\Pi'_\gamma}\hat{p}(\mu).\end{align*}

We know that any $p\in \Pi'_{\gamma}$ is of the form $e^{\epsilon_1}_{i_1}\tilde{p}e^{\epsilon_2}_{i_2}$, where $e^{\epsilon_1}_{i_1}$ and $e^{\epsilon_2}_{i_2}$ cross $\gamma$ ($e_{i_1}$ contains $\delta_1$ and $e_{i_2}$ contains $\delta_3$). Note that $\tilde{p}$ cannot contain the same branch with the same orientation twice since then it would contain a non-trivial loop which is impossible (a non-tight train path which contains a non-trivial loop is not minimal). Hence $\mu(\gamma)$ is piecewise linear since $\Pi'_\gamma$ is finite and for each of these train paths $\hat{p}(\mu)$ is piecewise linear by Lemma~\ref{piecewisemeasure}. Therefore the map $\mathcal{W}^+(\tau)\to \mathcal{S}_n$ is piecewise linear.
\end{proof}

Next, we shall illustrate Lemma \ref{piecewiselin} in the following example:
\begin{example}\label{trainmbedding}

Consider the $4$-braid $\beta=\sigma_1\sigma_{2}\sigma^{-1}_3$ on $D_4$. A standard embedding of the invariant train track $\tau$ of $\beta$ with respect to Dynnikov arcs is as depicted in Figure 8. Let $a,b,c,d$ and $m_1,m_2,m_3,m_4,m_5,m_6,m_7$ denote the measures on the main and infinitesimal branches of $\tau$. We first observe that the measures on the infinitesimal branches of $\tau$ are determined by  $a,b,c,d$ since the switch conditions give 
\begin{align*}
m_1&=a/2,~~ m_2=b/2,~~m_3=(c+d)/2,~~m_4=d/2\\*
m_5&=(a+b-c)/2,~~ m_6=(b+c-a)/2,~~ m_7=(a+c-b)/2.
\end{align*}
Since $D_n-\tau$ only has an unpunctured trigon and punctured monogons, $\tau$ is complete and $\rank(\tau)=4$.  We shall find the change of coordinate function $(a,b,c,d)\mapsto (a_1,a_2,b_1,b_2)$ from train track coordinates to Dynnikov coordinates.

\begin{figure}
\centering
\psfrag{a}[tl]{$\scriptstyle{a}$}
\psfrag{b}[tl]{$\scriptstyle{b}$}
\psfrag{c}[tl]{$\scriptstyle{c}$}
\psfrag{d}[tl]{$\scriptstyle{d}$}
\psfrag{b2}[tl]{$\scriptstyle{\beta_2}$}
\psfrag{b1}[tl]{$\scriptstyle{\beta_1}$}
\psfrag{b3}[tl]{$\scriptstyle{\beta_3}$}
\psfrag{m1}[tl]{$\scriptstyle{m_1}$}
\psfrag{m2}[tl]{$\scriptstyle{m_2}$}
\psfrag{m3}[tl]{$\scriptstyle{m_3}$}
\psfrag{m4}[tl]{$\scriptstyle{m_4}$}
\psfrag{x}[tl]{$\scriptstyle{\alpha_1}$}
\psfrag{y}[tl]{$\scriptstyle{\alpha_2}$}
\psfrag{z}[tl]{$\scriptstyle{\alpha_3}$}
\psfrag{t}[tl]{$\scriptstyle{\alpha_4}$}
\psfrag{l}[tl]{$\scriptstyle{m_7}$}
\psfrag{n}[tl]{$\scriptstyle{m_5}$}
\psfrag{m}[tl]{$\scriptstyle{m_6}$}
\psfrag{p1}[tl]{$\scriptstyle{p_1}$}
\psfrag{p2}[tl]{$\scriptstyle{p_2}$}
\psfrag{d}[tl]{$\scriptstyle{d}$}
\includegraphics[width=0.75\textwidth]{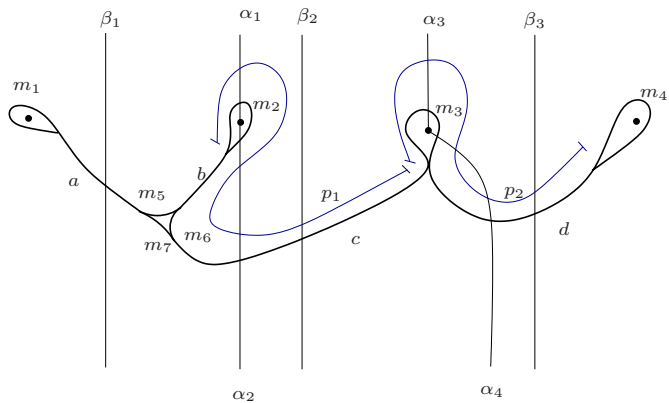}
\caption{A standard embedding of $\tau$ with respect to Dynnikov arcs}\label{traintrack2s}
\end{figure}

\begin{figure}
\centering
\psfrag{a}[tl]{$\scriptstyle{a}$}
\psfrag{a1}[tl]{$\scriptstyle{\alpha_1}$}
\psfrag{a2}[tl]{$\scriptstyle{\alpha_2}$}
\psfrag{b}[tl]{$\scriptstyle{b}$}
\psfrag{c}[tl]{$\scriptstyle{c}$}
\psfrag{m2}[tl]{$\scriptstyle{m_2}$}
\psfrag{m7}[tl]{$\scriptstyle{m_7}$}
\psfrag{m5}[tl]{$\scriptstyle{m_5}$}
\psfrag{m6}[tl]{$\scriptstyle{m_6}$}
\psfrag{m}[tl]{$\scriptstyle{m_6}$}
\psfrag{p1}[tl]{$\scriptstyle{p_1}$}
\psfrag{p2}[tl]{$\scriptstyle{p_2}$}
\psfrag{d}[tl]{$\scriptstyle{d}$}
\includegraphics[width=0.65\textwidth]{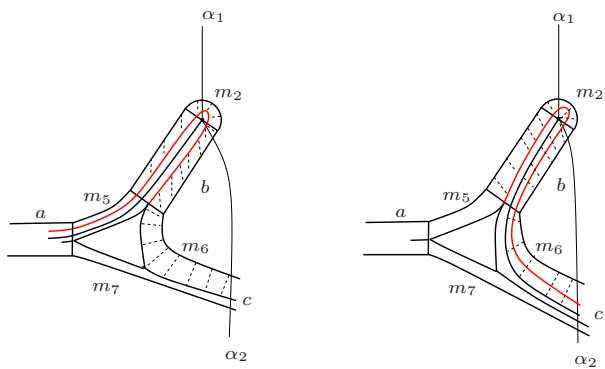}
\caption{ $\hat{p_1}(\mu)=m_6$ on the left and $\hat{p_1}(\mu)=b/2$ on the right}\label{traintrack2t}
\end{figure}

\begin{figure}
\centering
\psfrag{a}[tl]{$\scriptstyle{a}$}
\psfrag{m2}[tl]{$\scriptstyle{m_2}$}
\psfrag{d}[tl]{$\scriptstyle{d}$}
\psfrag{b}[tl]{$\scriptstyle{b}$}
\psfrag{c}[tl]{$\scriptstyle{c}$}
\psfrag{m3}[tl]{$\scriptstyle{m_3}$}
\psfrag{d+c}[tl]{$\scriptstyle{c+d}$}
\psfrag{m5}[tl]{$\scriptstyle{m_5}$}
\psfrag{m6}[tl]{$\scriptstyle{m_6}$}
\psfrag{m}[tl]{$\scriptstyle{m_6}$}
\psfrag{p1}[tl]{$\scriptstyle{p_1}$}
\psfrag{p2}[tl]{$\scriptstyle{p_2}$}
\psfrag{d}[tl]{$\scriptstyle{d}$}
\includegraphics[width=0.55\textwidth]{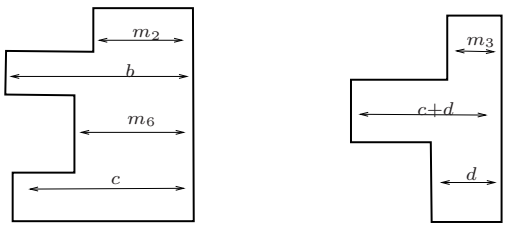}
\caption{$K_{p_1}$ and $K_{p_2}$}\label{schemaex2}
\end{figure}

We have $\beta_1=a$, $\beta_2=c$ and $\beta_3=d$ since $\Pi'_{\beta_i}=\emptyset$ for $i=1,2,3$. 
Hence, $$b_1=\frac{a-c}{2} ~~\text{and}~~ b_2=\frac{c-d}{2}.$$
We also have $\Pi'_{\alpha_1}=\Pi'_{\alpha_3}=\emptyset$, 
 and $\Pi'_{\alpha_2}=\{p_1\},~~\Pi'_{\alpha_4}=\{p_2\}$ where $p_1$ and $p_2$ are as depicted in Figure \ref{traintrack2s}.
 
 We observe from Figure \ref{traintrack2t} and Figure \ref{schemaex2} that

\begin{align*}
\hat{p_1}(\mu)=\min(m_2,m_6)~~\text{and}~~\hat{p_2}(\mu)=\min(d,m_3)
\end{align*}

We have that
$$\alpha_1=m_2=b/2 ~~\text{and}~~\alpha_2=m_2+c-2\hat{p_1}(\mu).$$
Hence  
\begin{align*}
\alpha_2= \frac{b}{2}+c-\min(b,b+c-a)&=\max(a,c)-\frac{b}{2},
\end{align*}
and so,
\begin{align*}
a_1=\frac{\alpha_2-\alpha_1}{2}=\frac{\max(a,c)-b}{2}
\end{align*}

Similar computations give $a_2=\frac{\alpha_4-\alpha_3}{2}=\frac{\max(-c,-d)}{2}$, and hence we get

\begin{align*}
(a_1,a_2,b_1,b_2)=\left(\frac{\max(a,c)-b}{2},\frac{\max(-c,-d)}{2},\frac{a-c}{2},\frac{c-d}{2}\right).
\end{align*}
\end{example}

\subsection{The spectrum of a Dynnikov matrix and a train track transition matrix}\label{mainresults}
The aim of this section is to compare the spectra of Dynnikov matrices with the spectra of the train track transition matrices of $\beta\in B_n$.

\subsection*{Case 1: $(\mathcal{F}^u,\mu^u)$ has only unpunctured $3$-pronged and punctured $1$-pronged singularities}\label{section2part2}
Let $\beta\in B_n$ be a pseudo\,-Anosov braid with unstable invariant foliation $(\mathcal{F}^u,\mu^u)$ and an invariant train track $\tau$ with transition matrix $T$. Let $(\mathcal{F}^u,\mu^u)$ have only unpunctured $3$-pronged and punctured $1$-pronged singularities.

\begin{prooftheorem31}
\textnormal{Since $(\mathcal{F}^u,\mu^u)$ has only unpunctured $3$-pronged and punctured $1$-pronged singularities, $\tau$ is complete. That is, $\mathcal{W}^+(\tau)$ has dimension $2n-4$. Therefore, $\mathcal{MF}(\tau)$ is a chart on $\mathcal{MF}_n$ \cite{mosher1, penner}.
By Lemma \ref{perronfoliation}, the eigenvector $v$ associated with the dilatation $\lambda>1$ is a transverse measure on $\tau$ with $(\mathcal{F}^u,\mu^u)=~\phi_{\tau}(v)$. Furthermore, the entries of $v$ are strictly positive. Therefore, $\mathcal{W}^+(\tau)$ and $\mathcal{MF}(\tau)$ are neighbourhoods of $v$ and $(\mathcal{F}^u,\mu^u)$ respectively.
Construct the pre-foliation $\mathcal{F}^*$ from $\tau$ as described in Section \ref{subsection21}.
Because none of the prongs of the pre-foliation $\mathcal{F}^*$ are connected by Lemma \ref{perronfoliation}, it follows from Remark~\ref{notconnected} that there is a neighbourhood $U$ of $v\in \mathcal{W}^+(\tau)$ on which the change of coordinate function $L=\rho\circ\phi_{\tau}$ from train track coordinates to Dynnikov coordinates is linear. Write $\mathcal{R}=L(U)\subseteq\mathcal{S}_n$ which is a neighbourhood of $L(v)$. We have the following commutative diagram}:

\begin{align*}
\begin{CD}
\mathcal{W}^{+}(\tau) @>T>>\mathcal{W}^{+}(\tau)\\
@V\phi_{\tau}  VV @V\phi_{\tau} VV\\
\mathcal{MF}(\tau) @>\beta>>\mathcal{MF}(\tau)\\
@V \rho VV @V \rho VV\\
\mathcal{R}\subseteq\mathcal{S}_n @>F>>\mathcal{S}_n\\
\end{CD}
\end{align*}
\textnormal{Then $F|_{\mathcal{R}}=D=L\circ T\circ L^{-1}$ is linear and isospectral to $T$.}

\end{prooftheorem31}\vspace{-7mm}\qed

\begin{prooflemma33}
\textnormal{Let~$\{e_i\}_{1\le i\le N}$ be the oriented branches of~$\tau$. Take a
copy $e_i'$ of each~$e_i$ and endow it with the opposite
orientation. The lifted train track~$\tilde{\tau}$ is obtained by
gluing together the branches $e_i$ and $e_i'$ following the pattern of
the original train track~$\tau$, but in such a way that the
orientations of all of the branches at each switch are consistent. By
construction, the edge path~$\tilde{f}(e_i)$ is obtained from the edge
path~$f(e_i)$ by replacing each occurence of $\overline{e_j}$ with
$e_j'$; and similarly, the edge path~$\tilde{f}(e_i')$ is obtained from
the edge path~$\overline{f(e_i)}$ by replacing each occurence of
$\overline{e_j}$ with $e_j'$}.
\textnormal{Let~$A_{ij}$ be the number of occurences of $e_i$ in $\tilde{f}(e_j)$
(that is, the number of occurences of $e_i$ in $f(e_j)$),
which by construction is equal to the number of occurences of $e_i'$
in $\tilde{f}(e_j')$; and let $B_{ij}$ be the number of occurences
of~$e_i'$ in $\tilde{f}(e_j)$ (that is, the number of occurences of
$\overline{e_i}$ in $f(e_j)$), which by construction is equal to the
number of occurences of $e_i$ in $\tilde{f}(e_j')$. Hence the lifted
transition matrix~$\tilde{T}$ is of the form}
\[
\tilde{T}=\left(
\begin{array}{cc}
A & B \\ B & A
\end{array}
\right),
\]
\textnormal{where $A+B=T$ (and we have restricted to the main branches $e_i$
($1\le i\le k$) and their copies $e_i'$).
Hence we have}
\begin{align*}
\chi(\tilde{T})=\left|xI_{2k}-\tilde{T}\right|
&=\left|
\begin{array}{cc}
xI_k-A & -B \\
-B & xI_k-A \end{array}\right|
=\left|
\begin{array}{cc}
xI_k-A & xI_k-T \\
-B & xI_k-T \end{array}\right|\\
&=\left|
\begin{array}{cc}
xI_k-A+B & 0_k \\
-B & xI_k-T \end{array}\right|\\
&=\left|xI_k-A+B\right|\left|xI_k-T\right|
\end{align*}
\textnormal{That is, the set of eigenvalues of $\tilde{T}$ is the union of the set of eigenvalues of $T$ and the set of eigenvalues of $A-B$. It remains to show that the eigenvalues of $A-B$ are roots of unity}.

\textnormal{Now for each~$m\ge 1$, let $A^{(m)}_{ij}$ denote the number of
occurences of $e_i$ in $f^m(e_j)$, and $B^{(m)}_{ij}$ denote the
number of occurences of $\overline{e_i}$ in $f^m(e_j)$. A
straightforward induction shows that the matrix $A^{(m)}$ is the sum
of all products of~$m$ copies of~$A$ and~$B$ having an even number
of~$B$s, and $B^{(m)}$ is the sum of all products of~$m$ copies of~$A$
and~$B$ having an odd number of~$B$s: therefore $A^{(m)}-B^{(m)} =
(A-B)^m$.} 
\textnormal{Let~$m$ be such that $f^m$ fixes all of the prongs of~$\tau$. Then for
each $e_i$, the initial and terminal points of $e_i$ and of $f^m(e_i)$
are the same. Since each real branch disconnects~$\tau$, it follows
that $A^{(m)}_{ij}=B^{(m)}_{ij}$ for all $i\not=j$, and
$A^{(m)}_{ii}=B^{(m)}_{ii}+1$ for all~$i$ (the number of times that
$f^m(e_i)$ crosses~$e_i$ in the positive direction is one more than
the number of times it crosses in the negative direction). That is
\[(A-B)^m = A^{(m)}-B^{(m)} = Id,\]
so that all of the eigenvalues of $A-B$ are roots of unity as required.}
\end{prooflemma33}\vspace{-7mm}\qed 

\begin{corollary}\label{anycomplete}
Let $[f]\in \MCG(D_n)$ be a pseudo\,-Anosov  isotopy class with unstable invariant foliation $(\mathcal{F}^u,\mu^u)$ and dilatation $\lambda>1$. Let  $\tau$  be any complete invariant train track with associated transition matrix $T$. Then $T$ and $D$ are isospectral up to  roots of unity and zeros.
\end{corollary}

\begin{proof}
If $f: D_n \to D_n$ is a pseudo\,-Anosov  homeomorphism it lifts to a pseudo\,-Anosov  homeomorphism $\tilde{f}: M\to M$ where $M$ is the orientation double cover~\cite{R99}. Pick a regular invariant train track $\tau_r$ and an arbitrary invariant train track $\tau$ of $f: D_n \to D_n$ with associated transition matrices $T_r$ and $T$. Given two matrices $A$ and $B$, write $A\sim B$ if $A$ and $B$ are isospectral up to roots of unity and zeros.  Then, $D\sim T_r$ by Theorem \ref{firstthm}, $T_r\sim \tilde{T}_r$ by Lemma \ref{eski}, $\tilde{T}_r\sim \tilde{T}$ by Theorem \ref{rykken} and $\tilde{T}\sim T$ by Lemma \ref{eski}. Therefore, $D\sim T$.
\end{proof}

\begin{example}
The $4$-braid $\beta=\sigma_1 \sigma^{-1}_2\sigma^3_3\sigma_2\sigma_1\sigma^{-1}_2$ has an invariant train track as depicted in Figure \ref{complete2} with associated transition matrix 

$$
T=\left[ \begin {array}{cccc} 2&0&2&1\\ \noalign{\medskip}2&0&3&1
\\ \noalign{\medskip}1&1&2&0\\ \noalign{\medskip}1&0&4&0\end {array}
 \right], 
 $$
and the coordinates of the eigenvector of  $T$ corresponding to the Perron-Frobenius eigenvalue $\lambda=4.61158$ are given by 
$$( 0.50135, 0.59215, 0.41871, 0.47190).$$

$(\mathcal{F}^u,\mu^u)$ is in the interior of a Dynnikov region $\mathcal{R}$ and the action on this region is given by the Dynnikov matrix
$$
D=\left[ \begin {array}{cccc} 5&-2&3&1\\ \noalign{\medskip}3&0&1&-2
\\ \noalign{\medskip}1&-1&1&1\\ \noalign{\medskip}1&1&0&-2\end {array}
 \right]. $$

Both $D$ and $T$ have spectrum $$\{1+\sqrt{2}\pm \sqrt{2+2\sqrt{2}}, 1-\sqrt{2}\pm i\sqrt{2\sqrt{2}-2}\}.$$
\end{example}

\begin{figure}[h!]
\centering
\includegraphics[width=0.6\textwidth]{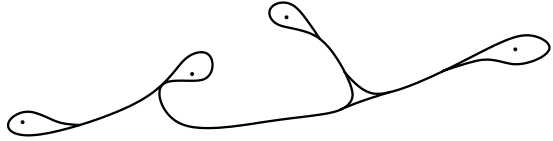}
\caption{Invariant train track for $\sigma_1 \sigma^{-1}_2\sigma^3_3\sigma_2\sigma_1\sigma^{-1}_2$}\label{complete2}
\end{figure}

\subsection*{Case 2: $(\mathcal{F}^u,\mu^u)$ has singularities other than punctured $1$-pronged and unpunctured $3$-pronged singularities}\label{section3part2}

 This section studies the case where the invariant unstable measured foliation $(\mathcal{F}^u,\mu^u)$ has other than punctured $1$-pronged and unpunctured $3$-pronged singularities. In this case $\tau$ is not complete and since $\rank(\tau)<2n-4$,  $\mathcal{MF}(\tau)$ does not define a chart on $\mathcal{MF}_n$. We shall study this case considering the two possibilities: first, where the prongs of the invariant foliations  are fixed by $\beta$; and second, where they are permuted non-trivially.

	If $\beta$ fixes the prongs, we shall see that every Dynnikov matrix is isospectral to $T$ up to some eigenvalues $1$.
%Let $\mathcal{R}_i$~$(1\leq i\leq r)$ be a Dynnikov region with associated Dynnikov matrix $D_i$ 
	 If $\beta$ permutes the prongs non-trivially, then for some power $m$, $\beta^m$ fixes the prongs and it follows that every Dynnikov matrix for $\beta^m$ is isospectral to $T^m$ up to some eigenvalues $1$ and zeros. 
	 
	 	 However, since the induced action of $\beta^m$ on $\mathcal{PS}_n$ is a product of several Dynnikov matrices, we can not conclude in the permuted prongs case that a Dynnikov matrix  $D_i$ and $T$ are isospectral up to roots of unity. 
	 	 
The main tool to prove our results will be to extend non-complete train tracks to those which are complete. We shall use two basic moves  \emph{pinching} and \emph{diagonal extension} \cite{mosher1,penner} as described as follows.

\begin{mydef}[\textbf{Pinching unpunctured $\bf t$-gons}]\label{ref:pinching}
Let $\tau$ be a train track with an unpunctured $t$-gon $P$, where $t\geq 4$. Let $e_1,\dots,e_t$ denote the (infinitesimal) edges of $P$. 
\emph{Pinching} across $e_i$ is a move which constructs a new train track $\tau'>\tau$ by pinching together the two edges $e_{i-1}$ and $e_{i+1}$ adjacent to $e_i$\footnote{Here and in what follows, indices are taken modulo $t$}.  See Figure~\ref{yenipinching}. The train track $\tau'$ has three additional edges denoted $e'_{i-1}$, $e'_{i+1}$ and $\epsilon$: in place of the $t$-gon $P$ it has a $(t-1)$-gon and a trigon. The function $\psi_{e_i}:\mathcal{W}(\tau)\to \mathcal{W}(\tau')$ is defined as follows.

If $w=(w_1,\dots,w_t,w_{t+1},\dots,w_k)\in \mathcal{W}(\tau)$, then $\psi_{e_i}(w)$ gives weights $w_{i-1}$ to $e'_{i-1}$, $w_{i+1}$ to $e'_{i+1}$, $w_{i-1}+w_{i+1}$ to $\epsilon$ and $w_j$ to $e_j$ for $1\leq j\leq k$. We remark that if every component of $w$ is positive, then the same is true for $\psi_{e_i}(w)$.

\end{mydef}
\begin{figure}[h!]
\centering
\psfrag{w1}[tl]{$\scriptstyle{e_1}$} 
\psfrag{5}[tl]{$\scriptstyle{\psi_{e_2}}$} 
\psfrag{w3}[tl]{$\scriptstyle{e_3}$} 
\psfrag{w2}[tl]{$\scriptstyle{e_2}$} 
\psfrag{w1p}[tl]{$\scriptstyle{e'_1}$}
\psfrag{w3p}[tl]{$\scriptstyle{e'_3}$}
\psfrag{w4}[tl]{$\scriptstyle{e_4}$}
\psfrag{a1}[tl]{$\scriptstyle{a_1}$} 
\psfrag{a2}[tl]{$\scriptstyle{a_2}$}
\psfrag{a4}[tl]{$\scriptstyle{a_4}$}
\psfrag{a3}[tl]{$\scriptstyle{a_3}$}
\psfrag{a5}[tl]{$\scriptstyle{a_5}$}
\psfrag{e}[tl]{$\scriptstyle{\epsilon}$}
%\psfrag{w1+w4}[tl]{$\scriptstyle{\epsilon=w_1+w_4}$} 
\psfrag{w5}[tl]{$\scriptstyle{e_5}$}
%\psfrag{(a+b-c)/2}[tl]{$\scriptstyle{\frac{a_1+a_2-a_3}{2}}$} 
%\psfrag{(a+c-b)/2}[tl]{$\scriptstyle{\frac{a_1+a_3-a_2}{2}}$} 
%\psfrag{(c+b-a)/2}[tl]{$\scriptstyle{\frac{a_3+a_2-a_1}{2}}$} 
\includegraphics[width=0.5\textwidth]{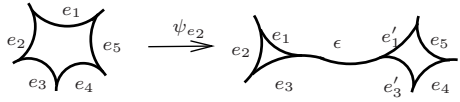}
\caption{Pinching move (across $e_2$) on an unpunctured $5$-gon}\label{yenipinching}
\end{figure}
\begin{figure}[h!]
\centering
\psfrag{e1}[tl]{$\scriptstyle{e_1}$} 
\psfrag{e2}[tl]{$\scriptstyle{e_2}$} 
\psfrag{e3}[tl]{$\scriptstyle{e_3}$} 
\psfrag{1}[tl]{$\scriptstyle{\epsilon}$}
%\psfrag{a}[tl]{$\scriptstyle{a_1}$} 
%\psfrag{b}[tl]{$\scriptstyle{a_2}$}
%\psfrag{c}[tl]{$\scriptstyle{a_3}$}
\psfrag{2}[tl]{$\scriptstyle{e'_1}$}
\psfrag{3}[tl]{$\scriptstyle{e''_1}$} 
\psfrag{4}[tl]{$\scriptstyle{\epsilon_2}$} 
\includegraphics[width=0.6\textwidth]{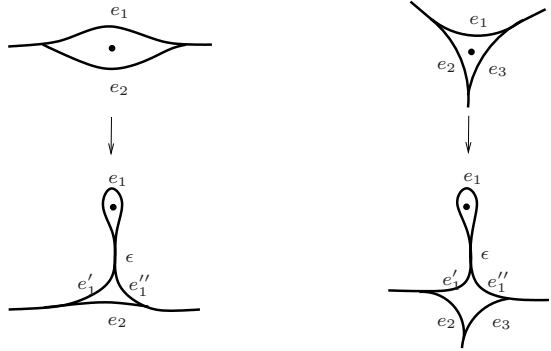}
\caption{Pinching move (of $e_1$) on a punctured bigon and a punctured trigon}\label{pinchingmoves2}
\end{figure}

\begin{mydef}[\textbf{Pinching punctured $\bf t$-gons}]\label{pinch2} 
Let $\tau$ be a train track with a punctured $t$-gon $P$ where $t\geq 2$. Let $e_1,\dots, e_t$ denote the (infinitesimal) edges of~$P$. 
\emph{Pinching} of $e_i$ is a move which constructs a new train track $\tau'>\tau$ by pinching $e_i$ to itself  around the puncture as depicted in Figure~\ref{pinchingmoves2}. The train track $\tau'$ has three additional edges denoted $\epsilon, e'_i, e_i''$: in place of the punctured $t$-gon, it has an unpunctured $(t+1)$-gon and a punctured monogon.

	 The function $\psi_{e_i}:\mathcal{W}(\tau)\to \mathcal{W}(\tau')$ is given as follows.
	
		If $w=(w_1,\dots,w_t,w_{t+1},\dots w_k)\in \mathcal{W}(\tau)$, $\psi_{e_i}(w)$ gives weights $2w_i$ to $\epsilon$, $w_i$ to $e'_i$ and $e''_i$, and $w_j$ to $e_j$ for $1\leq j\leq k$. We remark again that if every component of $w$ is positive, then the same is true for $\psi_{e_i}(w)$. 
\end{mydef}

\begin{mydef}
We say that a complete train track $\tau_p$ on $D_n$ is a \emph{pinching} of~$\tau$ if it is constructed from $\tau$ by a sequence of pinching moves.
\end{mydef}

\begin{remark}
Given a train track $\tau$, pinching each punctured $t$-gon with $t\geq 2$ yields a train track with only punctured monogons and unpunctured $t$-gons for $t\geq 3$. Pinching each unpunctured $t$-gon $t-3$ times then yields a pinching of~$\tau$. Observe that there are many different pinchings of $\tau$ (Figure~\ref{pinchings}). The main result about pinched train tracks in Lemma~\ref{chartrelation} doesn't depend on the choice of pinching.
\end{remark}

\begin{figure}[h!]
\centering
\psfrag{e1p}[tl]{$\scriptstyle{e'_1}$} 
\psfrag{e2p}[tl]{$\scriptstyle{e'_2}$} 
\psfrag{e3p}[tl]{$\scriptstyle{e'_3}$} 
\psfrag{e4p}[tl]{$\scriptstyle{e'_4}$} 
\psfrag{e1}[tl]{$\scriptstyle{e_1}$} 
\psfrag{e2}[tl]{$\scriptstyle{e_2}$} 
\psfrag{e3}[tl]{$\scriptstyle{e_3}$} 
\psfrag{e4}[tl]{$\scriptstyle{e_4}$}
\psfrag{v1}[tl]{$\scriptstyle{v_1}$} 
\psfrag{v2}[tl]{$\scriptstyle{v_2}$} 
\psfrag{v3}[tl]{$\scriptstyle{v_3}$} 
\psfrag{v4}[tl]{$\scriptstyle{v_4}$}
\includegraphics[width=0.4\textwidth]{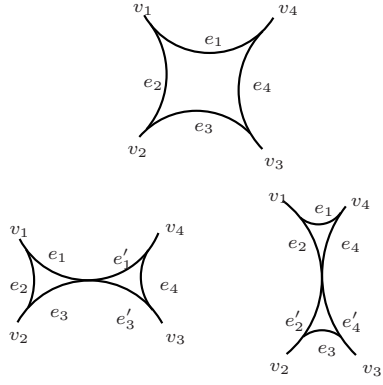}
\caption{Two different pinchings of an unpunctured $4$-gon}\label{pinchings}
\end{figure}

 Therefore,  pinching constructs a complete train track $\tau_p$ from a non-complete one $\tau$ in such a way that $\tau<\tau_p$, with the important feature that a strictly positive measure on $\tau$ induces a strictly positive measure on $\tau_p$. Hence, $\mathcal{MF}(\tau_p)$ defines a chart on $\mathcal{MF}_n$ which contains $(\mathcal{F}^u,\mu^u)$ in its interior. However, it should be noted that if $\tau$ is an invariant train track for $\beta$, $\tau_p$ will not be invariant unless relevant prongs of $(\mathcal{F}^u,\mu^u)$ are fixed by $\beta$. Therefore, we need a set of charts that fit nicely in $\mathcal{MF}(\tau_p)$ with the property that the action in each of them is described explicitly.
 
We shall use the \emph{diagonal extension} move to describe such charts. Diagonal extension gives a collection of  \emph{diagonally extended} train tracks $\tau_i$ in such a way that $\tau<\tau_i$. The disadvantage of diagonal extension is that a strictly positive measure on $\tau$ induces zero measure on the additional branches of $\tau_i$ and hence $(\mathcal{F}^u,\mu^u)$ is on the boundary of each $\mathcal{MF}(\tau_i)$. However, Lemma \ref{chartrelation} gives that the charts $\mathcal{MF}(\tau_i)$ fit together nicely and have union $\mathcal{MF}(\tau_p)$: moreover for each $i$ there is some $j$ such that $\beta(\mathcal{MF}(\tau_i))=\mathcal{MF}(\tau_j)$, and this action can be simply described with respect to appropriate bases.

\begin{mydef}[\textbf{Diagonal extension on unpunctured $\bf t$-gons}]\label{ref:extension}

Let $\tau$ be a train track with an unpunctured $t$-gon $P$ where $t\geq 4$. Let $v_1,\dots, v_t$ denote the vertices of $P$. \emph{Diagonal extension} of $P$ is a move which constructs a new train track $\tau'>\tau$ by adding $t-3$ branches (with disjoint interiors) inside $P$ such that each additional branch joins two non-consecutive vertices $v_i$ and $v_j$ and is tangent to the (infinitesimal) edges of $P$ at these vertices. See Figure~\ref{adding1}. The train track $\tau'$ has $t-3$ additional edges denoted $\epsilon_{ij}$ for appropriate choices of $i$ and $j$ with $|i-j|>1$: in place of the $t$-gon, it has $t-2$ unpunctured trigons. The function $\psi:\mathcal{W}(\tau)\to \mathcal{W}(\tau')$ is given as follows.
	
		If $w=(w_1,\dots, w_k)\in \mathcal{W}(\tau)$, $\psi(w)$ gives zero weights to each $\epsilon_{ij}$, and weight $w_i$ to $e_i$ for $1\leq i\leq k$. 

\end{mydef}
\begin{figure}[h!]
\centering
\psfrag{a1}[tl]{$\scriptstyle{v_1}$} 
\psfrag{a2}[tl]{$\scriptstyle{v_2}$} 
\psfrag{a3}[tl]{$\scriptstyle{v_3}$} 
\psfrag{a4}[tl]{$\scriptstyle{v_4}$}
\psfrag{a5}[tl]{$\scriptstyle{v_5}$} 
\psfrag{t}[tl]{$\scriptstyle{\tau}$}
\psfrag{e2}[tl]{$\scriptstyle{\epsilon_{35}}$}
\psfrag{e1}[tl]{$\scriptstyle{\epsilon_{25}}$} 
\psfrag{t2}[tl]{$\scriptstyle{\tau'}$}
\psfrag{t1}[tl]{$\scriptstyle{\tau_1=\tau_3}$}
\includegraphics[width=0.55\textwidth]{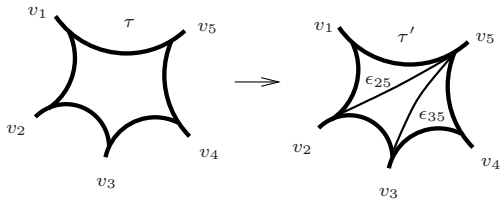}
\caption{Diagonal extension on an unpunctured $5$-gon}\label{adding1}
\end{figure}

\begin{mydef}[\textbf{Diagonal extension on punctured $\bf t$-gons}]

Let $\tau$ be a train track with a punctured $t$-gon $P$ where $t\geq 2$. Let $v_1,\dots, v_t$ denote the vertices of $P$. \emph{Diagonal extension} of $P$ is a move which constructs a new train track $\tau'>\tau$ by first adding a  branch $\epsilon_{ii}$ which encircles the puncture with both end points at a single vertex $v_i$ (so that $P$ is divided into a punctured monogon and an unpunctured $(t+1)$-gon); and then adding $t-2$ additional branches to divide the $(t+1)$-gon into $t-1$ trigons as in the unpunctured case. See Figure~\ref{adding2}. The train track $\tau'$ therefore has a punctured monogon and $t-1$ unpunctured trigons in place of the punctured $t$-gon $P$. The function $\psi:\mathcal{W}(\tau)\to \mathcal{W}(\tau')$ is given as follows.
	
		If $w=(w_1,\dots, w_k)\in \mathcal{W}(\tau)$, $\psi(w)$ gives weight $w_j$ to $e_j$ for $1\leq j\leq k$, and weight zero to the other edges.

\end{mydef}

\begin{figure}[h!]
\centering
\psfrag{a1}[tl]{$\scriptstyle{v_1}$} 
\psfrag{a2}[tl]{$\scriptstyle{v_2}$} 
\psfrag{w3}[tl]{$\scriptstyle{w_3}$} 
\psfrag{w2}[tl]{$\scriptstyle{w_2}$} 
\psfrag{w1}[tl]{$\scriptstyle{w_1}$} 
\psfrag{e1}[tl]{$\scriptstyle{\epsilon_1}$} 
\psfrag{e12}[tl]{$\scriptstyle{\epsilon_{13}}$} 
\psfrag{e2}[tl]{$\scriptstyle{\epsilon_{22}}$} 
\psfrag{a3}[tl]{$\scriptstyle{v_3}$} 
\psfrag{a4}[tl]{$\scriptstyle{v_4}$}
\psfrag{a5}[tl]{$\scriptstyle{v_5}$} 
\psfrag{t1}[tl]{$\scriptstyle{\tau_1}$}
\psfrag{t2}[tl]{$\scriptstyle{\tau_2}$}
\includegraphics[width=0.65\textwidth]{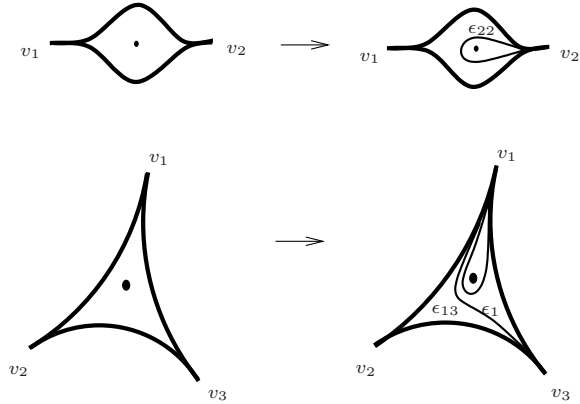}
\caption{Diagonal extension on a punctured bigon and a punctured trigon}\label{adding2}
\end{figure}
\begin{mydef}\label{diagonalextended}
We say that a complete train track $\tau'$ on $D_n$ is a \emph{diagonal extension} of $\tau$ if it is constructed from $\tau$ by a sequence of diagonal extensions. We write $\tau_1,\tau_2,\dots,\tau_{\xi}$ to denote the different diagonal extensions of $\tau$. 
\end{mydef}
\begin{remark}\label{catalan}
Note that the number of diagonal extensions of an unpunctured $t$-gon is $\xi=c_{t-2}$ where 
$$c_t=\binom{2t}{t}-\binom{2t}{t-1}$$
is the $t^\text{th}$ Catalan number, since the Catalan number gives the number of different ways to divide a polygon into triangles by joining its vertices with additional edges.

	 Similarly, the number of diagonal extensions of a punctured $t$-gon is $\xi=t \cdot c_{t-1}$: after adding the encircling branch $\epsilon_{ii}$ at vertex $v_i$ ($t$ choices), $P$ is divided into an unpunctured $(t+1)$-gon and a punctured monogon. Since there are $c_{t-1}$ different ways to divide an unpunctured $(t+1)$-gon into triangles the result follows.
	 
	Given a train track $\tau$, let $G$ denote the set of unpunctured polygons of $\tau$, and given $P\in G$ write $n_P$ for the number of vertices of $P$. Similarly, let $G'$ denote the set of punctured polygons of $\tau$, and given $P\in G'$ write $n_P$ for the number of vertices of $P$. 
Then the number of diagonal extensions of $\tau$ is given by $$\xi=\left(\displaystyle\prod_{P\in G}c_{n_P-2}\right)\cdot\left(\displaystyle\prod_{P\in G'}n_P\cdot c_{n_P-1}\right).$$
$\xi$ can be large for relatively simple train tracks and hence there can be many Dynnikov regions for braids on relatively few strings  

It will be seen from the proof of Theorem \ref{chartrelation} that there is a unique Dynnikov region that corresponds to a diagonal extension $\tau_i$ of $\tau$. Therefore, it follows that the number of Dynnikov regions is bounded above by the number of diagonal extensions of $\tau$ (which is given by the formula above).
\end{remark}

Since each pinching $\tau_p$ and diagonal extension $\tau_i$ of $\tau$ is complete, they define charts on $\mathcal{MF}_n$.  The following key lemma describes how these charts fit together. 
%Recall that $\mathcal{MF}(\tau)$ denotes the space of measured foliations carried by a train track $\tau$ and $\mathcal{W}^+(\tau)$ the corresponding space of transverse measures.

\begin{lem}\label{chartrelation}
Let $\tau$ be a regular invariant train track for $\beta$ with associated matrix $T:\mathcal{W}(\tau)\to\mathcal{W}(\tau)$. Let $\tau_p$ be a pinching of $\tau$, and let $\tau_1,\dots,\tau_{\xi}$ denote the diagonal extensions of $\tau$. Then,  

\begin{enumerate}[i.]
\item $\displaystyle\bigcup_{1\leq i\leq \xi} \mathcal{MF}(\tau_i)=\mathcal{MF}(\tau_p)$.
\item If $i\neq j$, then $\mathcal{MF}(\tau_i)$ and $\mathcal{MF}(\tau_j)$  intersect only on their boundaries.
\item For each $i$ there is some $j$ such that $\beta(\mathcal{MF}(\tau_i))=\mathcal{MF}(\tau_j)$, and the induced action of $\beta:\mathcal{W}(\tau_i)\to\mathcal{W}(\tau_j)$ is given by a matrix of the form $$\tilde{T}=\left[ \begin {array}{cc} T&X\\ \noalign{\medskip}0&{\it Id}
\end {array} \right]
$$
with respect to an appropriate choice of bases of $\mathcal{W}(\tau_i)$ and $\mathcal{W}(\tau_j)$.
\item For each $i$, the change of coordinate function $\rho\circ\phi_{\tau_i}:\mathcal{W}^+(\tau_i)\to\mathcal{S}_n$ is linear in a neighbourhood in $\mathcal{W}^+(\tau_i)$ of  $v^u=\phi^{-1}_{\tau_i}(\mathcal{F}^u,\mu^u)$.   
\end{enumerate}
\end{lem}
\begin{proof}
Assume first that every component of $D_n-\tau$ is a punctured monogon or unpunctured trigon, except for one unpunctured $t$-gon $P$ $(t\geq 4)$. Let $v_1,v_2,\dots,v_t$ denote the vertices of $P$. Let $\tau_p$ be a pinching of $\tau$ and $N$ denote a regular neighbourhood of the pinched $t$-gon. Let $a_1,\dots,a_t$ denote the \emph{gates} of $N$ (that is, the components of the subset of $\partial N$ which is not comprised of leaves). See Figure~\ref{regularpinchedgon}.

\begin{figure}[h!]
\centering
\psfrag{v1}[tl]{$\scriptstyle{v_1}$} 
\psfrag{v2}[tl]{$\scriptstyle{v_2}$} 
\psfrag{v3}[tl]{$\scriptstyle{v_3}$}
\psfrag{v4}[tl]{$\scriptstyle{v_4}$} 
\psfrag{v5}[tl]{$\scriptstyle{v_5}$} 
\psfrag{a1}[tl]{$\scriptstyle{a_1}$} 
\psfrag{a2}[tl]{$\scriptstyle{a_2}$} 
\psfrag{a3}[tl]{$\scriptstyle{a_3}$}
\psfrag{a4}[tl]{$\scriptstyle{a_4}$} 
\psfrag{a5}[tl]{$\scriptstyle{a_5}$} 
\psfrag{N}[tl]{$\scriptstyle{N}$}
\includegraphics[scale=0.7]{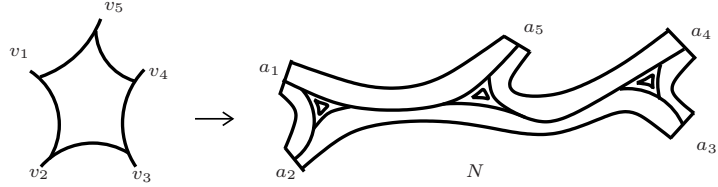}
\caption{The regular neighbourhood $N$ of a pinched unpunctured $5$-gon}\label{regularpinchedgon}
\end{figure}

 To each $(\mathcal{F},\mu)\in \mathcal{MF}(\tau_p)$ we associate the collection  of two element sets $\{j,k\}$ ($|j-k|>1$) such that $(\mathcal{F},\mu)$ has a leaf which enters $N$ through $a_j$ and exits through $a_k$. Denote this \emph{label set} $\Gamma(\mathcal{F},\mu)$ and observe that the cardinality $|\Gamma(\mathcal{F},\mu)|\leq t-3$ since leaves don't cross. 
 
 	Similarly, to each diagonal extension $\tau_i$ of $\tau$, we associate the two element sets $\{j,k\}$  such that $\tau_i$ has a branch joining $v_j$ to $v_k$. Denote this \emph{label set} $\Delta(\tau_i)$. It is clear that $(\mathcal{F},\mu)\in \mathcal{MF}(\tau_i)$ if and only if $\Gamma(\mathcal{F},\mu)\subseteq\Delta(\tau_i)$.

\begin{figure}[h!]
\centering
\psfrag{v1}[tl]{$\scriptstyle{v_1}$} 
\psfrag{v2}[tl]{$\scriptstyle{v_2}$} 
\psfrag{v3}[tl]{$\scriptstyle{v_3}$}
\psfrag{v4}[tl]{$\scriptstyle{v_4}$} 
\psfrag{v5}[tl]{$\scriptstyle{v_5}$} 
\psfrag{a1}[tl]{$\scriptstyle{a_1}$} 
\psfrag{a2}[tl]{$\scriptstyle{a_2}$} 
\psfrag{a3}[tl]{$\scriptstyle{a_3}$}
\psfrag{a4}[tl]{$\scriptstyle{a_4}$} 
\psfrag{a5}[tl]{$\scriptstyle{a_5}$} 
\psfrag{t1}[tl]{$\scriptstyle{\tau_1}$} 
\psfrag{t2}[tl]{$\scriptstyle{\tau_2}$} 
\psfrag{N}[tl]{$\scriptstyle{N}$}
\includegraphics[scale=0.7]{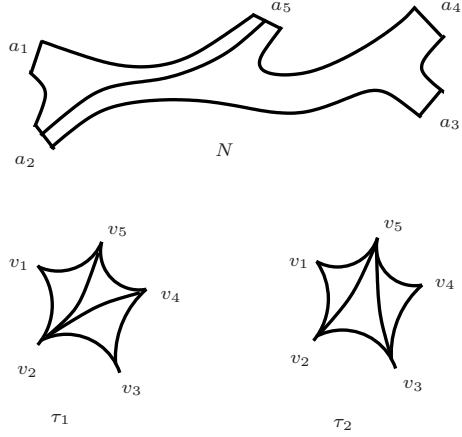}
\caption{If $\Gamma(\mathcal{F},\mu)=\{\{2,5\}\}$, then $(\mathcal{F},\mu)\in \mathcal{MF}(\tau_1)\cap \mathcal{MF}(\tau_2)$.}\label{choice}
\end{figure}

If $|\Gamma(\mathcal{F},\mu)|=t-3$, then there is a unique $\tau_i$ with $\Gamma(\mathcal{F},\mu)=\Delta(\tau_i)$; while if $|\Gamma(\mathcal{F},\mu)|<t-3$ then there are several $\tau_i$ with $\Gamma(\mathcal{F},\mu)\subset\Delta(\tau_i)$ and for each of these $\tau_i$, $\phi^{-1}_{\tau_i}(\mathcal{F},\mu)$ has some zero coordinates (see Figure \ref{choice}). This establishes that 

\begin{enumerate}[i.]
\item $\mathcal{MF}(\tau_p)\subseteq \displaystyle\bigcup^{\xi}_{i=1}{\mathcal{MF}(\tau_i)}$;
and
\item If $\tau_i\neq\tau_j$, then $\mathcal{MF}(\tau_i)$ and $\mathcal{MF}(\tau_j)$ can only intersect along their boundary.
\end{enumerate} 
 
 	  To show that $\displaystyle\bigcup^{\xi}_{i=1}\mathcal{MF}(\tau_i)\subseteq \mathcal{MF}(\tau_p)$, we observe that any two vertices of the pinched polygon of $\tau_p$ can be connected by a smooth path in $\tau_p$ and hence if $(\mathcal{F},\mu)$ is carried by any $\tau_i$, it is also carried by $\tau_p$.
 	  
 	   Next assume that every component of $D_n-\tau$ is a punctured monogon or unpunctured trigon, except for one punctured $t$-gon $P$ $(t\geq 2)$. Let $v_1,v_2,\dots,v_t$ denote the vertices of $P$. Let $\tau_p$ be a pinching of $\tau$ and $N$ denote a regular neighbourhood of the pinched $t$-gon. Label the gates $a_1,\dots,a_t$ of $N$ in anticlockwise cyclic order and let $l_i$ $(1 \leq i \leq t)$ denote the  leaf of $(\mathcal{F},\mu)$ in $\partial N$ which joins $a_{i-1}$ to $a_i$. See Figure~\ref{puncturedneigh}. 	  
 	    
\begin{figure}[h!]
\centering
\psfrag{v1}[tl]{$\scriptstyle{v_1}$} 
\psfrag{v2}[tl]{$\scriptstyle{v_2}$} 
\psfrag{v3}[tl]{$\scriptstyle{v_3}$}
\psfrag{a1}[tl]{$\scriptstyle{a_1}$} 
\psfrag{a2}[tl]{$\scriptstyle{a_2}$} 
\psfrag{a3}[tl]{$\scriptstyle{a_3}$}
\psfrag{l1}[tl]{$\scriptstyle{l_1}$} 
\psfrag{l2}[tl]{$\scriptstyle{l_2}$} 
\psfrag{l3}[tl]{$\scriptstyle{l_3}$}
\psfrag{N}[tl]{$\scriptstyle{N}$}
\includegraphics[scale=0.6]{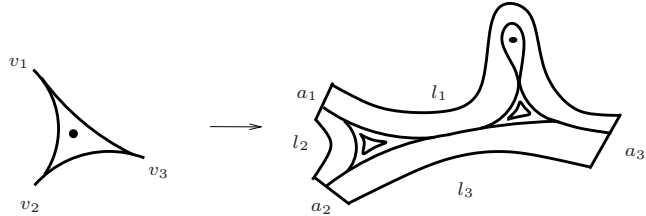}
\caption{The regular neighbourhood $N$ of a pinched punctured trigon}\label{puncturedneigh}
\end{figure}

 	The label set $\Gamma(\mathcal{F},\mu)$ of a measured foliation $(\mathcal{F},\mu)\in \mathcal{MF}(\tau_p)$ will consist of pairs $(j,k)\in \{1,\dots,t\}\times \{1,\dots,t\}$. This contrasts with the case for unpunctured $t$-gons, since there are two possible paths for leaves joining the gates $a_j$ and $a_k$, one on each side of the puncture. To describe this label set, first orient each gate $a_i$ and each leaf $l_i$ anticlockwise around $\partial N$. Then $(j,k)\in \Gamma(\mathcal{F},\mu)$ if and only if 
\begin{itemize}
\item there is a leaf segment $L$ of $(\mathcal{F},\mu)$ in $N$ which joins $a_j$ to $a_k$;
\item when $L$ is oriented from $a_j$ to $a_k$, the oriented loop  consisting of $L$ and a subset of $\partial N$ bounds a disk containing the puncture in its interior; and
\item $k\neq j+1$ (we don't include leaves which must necessarily be part of $N$).
\end{itemize}    
See Figure \ref{labelset1}. Notice that $(j,j)\in \Gamma(\mathcal{F},\mu)$ if and only if the leaf from the $1$-pronged singularity exits $N$ through $a_j$. Of course, it is possible that this leaf doesn't exit $N$ (e.g. if $(\mathcal{F},\mu)=(\mathcal{F}^u,\mu^u)$). Also, observe that the cardinality $|\Gamma(\mathcal{F},\mu)|\leq t-1$ since leaves don't cross.
 	
\begin{figure}[h!]
\centering
\psfrag{1}[tl]{$\scriptstyle{a_1}$} 
\psfrag{2}[tl]{$\scriptstyle{a_2}$} 
\psfrag{3}[tl]{$\scriptstyle{a_3}$}
\psfrag{4}[tl]{$\scriptstyle{a_4}$} 
\psfrag{5}[tl]{$\scriptstyle{a_5}$} 
\psfrag{l1}[tl]{$\scriptstyle{l_1}$} 
\psfrag{l2}[tl]{$\scriptstyle{l_2}$} 
\psfrag{l3}[tl]{$\scriptstyle{l_3}$}
\psfrag{l4}[tl]{$\scriptstyle{l_4}$} 
\psfrag{l5}[tl]{$\scriptstyle{l_5}$} 
\psfrag{N}[tl]{$\scriptstyle{N}$}
\includegraphics[scale=0.5]{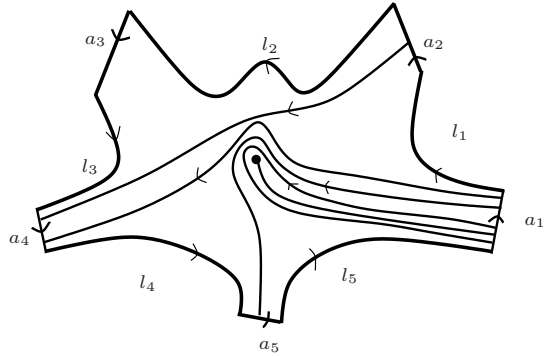}
\caption{The label set $\Gamma(\mathcal{F},\mu)$ is given by $\{(1,1),(1,5),(1,4),(2,4)\}$.}\label{labelset1}
\end{figure}
 	
 	To describe the label set $\Delta(\tau_i)$ for a diagonal extension $\tau_i$ of $\tau$, we label the vertices $v_1,\dots,v_t$ of $P$ in the anticlockwise cyclic order and put arrows on the edges of $P$ pointing from $v_j$ to $v_{j+1}$. For each additional branch, we place an arrow on the branch so that the loop composed of the branch and of edges of $P$ which encloses the puncture is oriented consistently. Then $\Delta(\tau_i)$ is the set of pairs $(j,k)$ such that there is an additional branch from $v_j$ to $v_k$. See Figure \ref{labelset2}.  It is clear that $(\mathcal{F},\mu)\in \mathcal{MF}(\tau_i)$ if and only if $\Gamma(\mathcal{F},\mu)\subseteq\Delta(\tau_i)$.

\begin{figure}[h!]
\centering
\psfrag{a1}[tl]{$\scriptstyle{v_1}$} 
\psfrag{a2}[tl]{$\scriptstyle{v_2}$} 
\psfrag{a3}[tl]{$\scriptstyle{v_3}$}
\psfrag{a4}[tl]{$\scriptstyle{v_4}$} 
\psfrag{a5}[tl]{$\scriptstyle{v_5}$} 
\psfrag{N}[tl]{$\scriptstyle{N}$}
\includegraphics[scale=1.1]{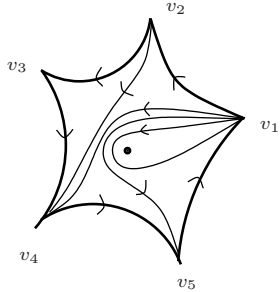}
\caption{The label set $\Delta(\tau_i)$ is given by $\{(1,1),(1,5),(1,4),(2,4)\}$}\label{labelset2}
\end{figure}

 	 If $|\Gamma(\mathcal{F},\mu)|=t-1$, then there is a unique $\tau_i$ with $\Gamma(\mathcal{F},\mu)=\Delta(\tau_i)$, while if $|\Gamma(\mathcal{F},\mu)|<t-1$ then there are several $\tau_i$ with $\Gamma(\mathcal{F},\mu)\subset\Delta(\tau_i)$ and for each of these $\tau_i$, $\phi^{-1}_{\tau_i}(\mathcal{F},\mu)$ has some zero coordinates. See Figure \ref{choice2}.

\begin{figure}[h!]
\centering
\psfrag{v1}[tl]{$\scriptstyle{v_1}$} 
\psfrag{v2}[tl]{$\scriptstyle{v_2}$} 
\psfrag{v3}[tl]{$\scriptstyle{v_3}$} 
\psfrag{a3}[tl]{$\scriptstyle{a_3}$}
\psfrag{a1}[tl]{$\scriptstyle{a_1}$} 
\psfrag{a2}[tl]{$\scriptstyle{a_2}$} 
\psfrag{t1}[tl]{$\scriptstyle{\tau_1}$} 
\psfrag{t2}[tl]{$\scriptstyle{\tau_2}$} 
\psfrag{ntp}[tl]{$\scriptstyle{N}$}
\includegraphics[scale=0.4]{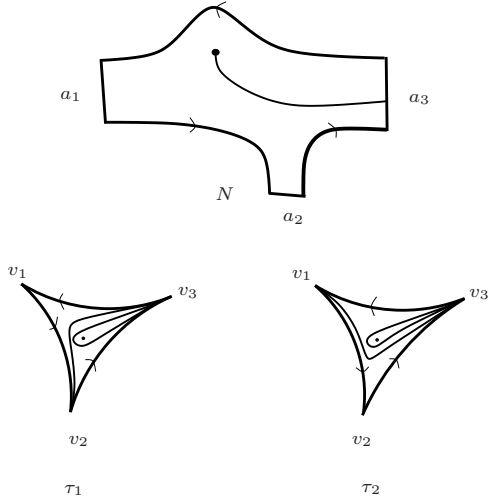}
\caption{If $\Gamma(\mathcal{F},\mu)=\{(3,3)\}$, then $(\mathcal{F},\mu)\in \mathcal{MF}(\tau_1)\cap \mathcal{MF}(\tau_2)$.}\label{choice2}
\end{figure}
This establishes that 
\begin{enumerate}[i.]
\item $\mathcal{MF}(\tau_p)\subseteq \displaystyle\bigcup^{\xi}_{i=1}{\mathcal{MF}(\tau_i)}$;
and
\item If $\tau_i\neq\tau_j$, then $\mathcal{MF}(\tau_i)$ and $\mathcal{MF}(\tau_j)$ can only intersect along their boundary.
\end{enumerate} 
 To show that $\displaystyle\bigcup^{\xi}_{i=1}\mathcal{MF}(\tau_i)\subseteq \mathcal{MF}(\tau_p)$, we observe that any two vertices of the pinched polygon of $\tau_p$ can be connected by a smooth path in $\tau_p$ and hence if $(\mathcal{F},\mu)$ is carried by any $\tau_i$, it is also carried by $\tau_p$.
 	  
 	  Therefore, we have proved the first two statements of the lemma in the case where $\tau$ has only one polygon which is not a punctured monogon or an unpunctured trigon. For the general case, we argue for each punctured and unpunctured polygon of $\tau$ in the same way as above and observe that if $(\mathcal{F},\mu)\in \mathcal{MF}(\tau_p)$, then there is a diagonal extension $\tau_i$ of $\tau$ so that $(\mathcal{F},\mu)\in \mathcal{MF}(\tau_i)$. Conversely, if  $(\mathcal{F},\mu)\in \mathcal{MF}(\tau_i)$ for some diagonal extension $\tau_i$ of $\tau$, then $(\mathcal{F},\mu)\in \mathcal{MF}(\tau_p)$ since any two vertices of each pinched polygon of $\tau_p$ can be connected by a smooth path. Also, if $(\mathcal{F},\mu)$ is carried by two diagonal extensions $\tau_i$ and $\tau_j$, then $\phi^{-1}_{\tau_i}(\mathcal{F},\mu)$ has some zero coordinates from the argument above and hence $\mathcal{MF}(\tau_i)$ and $\mathcal{MF}(\tau_j)$ can only intersect along their boundary. 
 	 
For the proof of the third statement, we first note that $\beta$ permutes the vertices of $\tau$. Hence, 
given a diagonal extension $\tau_i$ of $\tau$, the permutation on the vertices of $\tau$ sends each additional branch of $\tau_i$ onto another additional branch, and so gives another diagonal extension $\tau_j$ of $\tau$. Therefore, we have $\beta(\mathcal{MF}(\tau_i))=\mathcal{MF}(\tau_j)$. Then, $\beta:\mathcal{W}^+(\tau_i)\to\mathcal{W}^+(\tau_j)$ is described by the matrix $$\tilde{T}=\left[ \begin {array}{cc} T&X\\ \noalign{\medskip}0&{\it Id}
\end {array} \right]
$$
with respect to the natural coherent choice of bases of $\mathcal{W}^+(\tau_i)$ and $\mathcal{W}^+(\tau_j)$. We remark that if all components of $D_n-\tau$ are odd-gons then $\mathcal{W}^+(\tau_i)$ and $\mathcal{W}^+(\tau_j)$ have bases consisting of weights on the main branches of $\tau$ and the additional branches and hence $X$ is zero. If $D_n-\tau$ has an even-gon then $X$ can be non-zero since the bases of $\mathcal{W}^+(\tau_i)$ and $\mathcal{W}^+(\tau_j)$ consists of weights on edges which includes infinitesimal and additional ones. For some main branch $e_k$ of $\tau_i$, the corresponding weight $w_k$ is the sum of weights on some infinitesimal and additional branches and $f(e_k)$ may cover some basis elements.  

	 For the fourth statement,  we recall from (the proof of)	Lemma \ref{piecewiselin}, that the fact that each change of coordinates from train track coordinates $\mathcal{W}^+(\tau_i)$ to Dynnikov coordinates is only piecewise linear is a consequence of the piecewise linearity of the function $\hat{p}:\mathcal{W}^+(\tau_i)\to \mathbb{R}^+$, where $\hat{p}$ is a minimal non-tight train path with respect to some Dynnikov arc. The function $\hat{p}$ makes a transition from one linear region to another at measures for which some leaf which follows the train path $p$ connects two singularities (see the proof of Lemma \ref{piecewisemeasure}). 
	 
	 At $v^u=\phi^{-1}_{\tau_i}(\mathcal{F}^u,\mu^u)\in \mathcal{MF}(\tau_i)$ there are several such leaves connecting singularities but the choice of diagonal extension $\tau_i$ is precisely  a choice of the relative configurations of these leaves when the connection between singularities are broken, and therefore   $L:\mathcal{W}^+(\tau_i)\to\mathcal{S}_n$ is linear near $v^u$. 	
\end{proof}

Next we shall prove that when $\beta$ fixes the prongs of $\tau$, then every Dynnikov matrix is isospectral to $T$ up to some eigenvalues $1$.

\begin{prooftheorem34}

\textnormal{Let $\tau_p$ be a pinching of $\tau$ and $\tau_i$ $(1\leq i \leq \xi)$ be the diagonal extensions of~$\tau$. Since each strictly positive measure on $\tau$ induces a strictly positive measure on $\tau_p$, $\mathcal{W}^+(\tau_p)$ and $\mathcal{MF}(\tau_p)$ are neighbourhoods of $v^u$ and $(\mathcal{F}^u,\mu^u)$ respectively. Furthermore, by Lemma~\ref{chartrelation}, $\displaystyle\bigcup_{1\leq i\leq \xi} \mathcal{MF}(\tau_i)=\mathcal{MF}(\tau_p)$ and for $i\neq j$, $\mathcal{MF}(\tau_i)$ and $\mathcal{MF}(\tau_j)$ intersect only on their boundaries. 
Since $\beta$ fixes the prongs at all singularities other than unpunctured $3$-pronged and punctured $1$-pronged singularities, for each $\tau_i$ we have $\beta(\mathcal{MF}(\tau_i))=\mathcal{MF}(\tau_i)$ and the induced action $\beta:\mathcal{W}^+(\tau_i)\to\mathcal{W}^+(\tau_i)$ is given by the matrix} $$\tilde{T}=\left[ \begin {array}{cc} T&X\\ \noalign{\medskip}0&{\it Id}
\end {array} \right].
$$\textnormal{
 	By the fourth statement of Lemma \ref{chartrelation}, $\rho\circ\phi_{\tau_i}:\mathcal{W}^+(\tau_i)\to \mathcal{S}_n$ is  linear in a neighbourhood in $\mathcal{W}^{+}(\tau_i)$ of $v^u=\phi^{-1}_{\tau_i}(\mathcal{F},\mu)$. Therefore, for each $\tau_i$ we have the following commutative diagram:}

\begin{align*}\label{diagramcomplete}
\begin{CD}
\mathcal{W}^{+}(\tau_i) @>\tilde{T}>>\mathcal{W}^{+}(\tau_i)\\
@V\phi_{\tau_i}  VV @V\phi_{\tau_i} VV\\
\mathcal{MF}(\tau_i) @>\beta>>\mathcal{MF}(\tau_i)\\
@V \rho VV @V \rho VV\\
\mathcal{S}_n @>F>>\mathcal{S}_n\\
\end{CD}
\end{align*}
\textnormal{where the change of coordinate function $L_i=\rho\circ\phi_{\tau_i}$ is linear in a neighbourhood $U_i\subseteq\mathcal{W}^{+}(\tau_i)$ of $v^u$. Let $(a^u,b^u)$ denote the Dynnikov coordinates of $(\mathcal{F}^u,\mu^u)$. For each $1\leq i\leq \xi$, write $\mathcal{R}_i=L_i(U_i)$: by the above, $\displaystyle\bigcup_{1\leq i\leq \xi}\mathcal{R}_i$ is a neighbourhood of $(a^u,b^u)$. Then in $\mathcal{R}_i$, $D_i=F|_{\mathcal{R}_i}=L_i\circ \tilde{T}\circ L_i^{-1}$ is linear and isospectral to $T$ up to some  eigenvalues~$1$. These matrices $D_i$ ($1\leq i\leq \xi$) are precisely the Dynnikov matrices for $\beta\in B_n$.}\end{prooftheorem34}\vspace{-7mm}\qed

	In fact, the proof of Theorem \ref{uniquematrix}  shows that if $D_n-\tau$ has only odd-gons all of the Dynnikov matrices are equal and hence there is only one Dynnikov region in the fixed-pronged case.

\begin{prooftheorem35}
\textnormal{We use the notation in the proof of Theorem \ref{thm1}. Let $k=\rank(\tau)$ and $N=2n-4$ be the dimension of $\mathcal{S}_n$. We first note that each $L_i$ is of the form $\displaystyle\left(L|X_i\right)$ for some fixed $N\times k$ matrix $L$, where $L$  is the change of coordinates from $\mathcal{W}^+(\tau)$ to $ \mathcal{S}_n$ on the hyperplane $\mathcal{MF}(\tau)$.  Then each $L^{-1}_i$ is of the form $\displaystyle\left(\frac{A}{Y_i}\right)$ for some fixed $k\times N$ matrix $A$ which gives the change of coordinates from the $k$-dimensional subspace of $\cS_n$ corresponding to $\mathcal{MF}(\tau)$ to $\mathcal{W}^+(\tau)$. Therefore we have,}
\begin{align*}
L^{-1}_iL_i=\left(\frac{A}{Y_i}\right)\left(L|X_i\right)
\end{align*}
\textnormal{Since $L^{-1}_iL_i=\it Id$, $AX_i$ and $Y_iL$ are zero matrices for all $i$. It follows that for any $i,j$ we have}

\begin{align*}
L^{-1}_jL_i=\left[ \begin {array}{cc} \it Id_k&0\\ \noalign{\medskip}0&{\it P_{ij}}
\end {array} \right].
\end{align*}

\textnormal{for some $(N-k)\times (N-k)$ matrix $P_{ij}$. In particular, $L^{-1}_jL_i$ commutes with}
$$
\tilde{T}=\left[ \begin {array}{cc} T&0\\ \noalign{\medskip}0&{\it Id}
\end {array} \right].
$$
\textnormal{Hence, $D_i=L_i\tilde{T}L^{-1}_i=L_j\tilde{T}L_j^{-1}=D_j$ for all $i$ and $j$.}
\end{prooftheorem35}\vspace{-7mm}\qed

\begin{corollary}\label{anytraintrack2}
Let $\beta\in \MCG(D_n)$ be a pseudo\,-Anosov  braid with unstable invariant foliation $(\mathcal{F}^u,\mu^u)$ and dilatation $\lambda>1$. Let  $\tau$  be any invariant train track with associated transition matrix $T$.  If $\beta$ fixes the prongs at all singularities other than unpunctured $3$-pronged and punctured $1$-pronged singularities, then any Dynnikov matrix is isospectral to $T$ up to roots of unity and zeros. Furthermore, if $D_n-\tau$ consists of only odd-gons then there is a unique Dynnikov matrix.
\end{corollary}

\begin{example}
Consider the $4$-braid $\gamma=\sigma^2_1 \sigma^2_2 \sigma_1 \sigma_2 \sigma^2_3\sigma_2 \sigma^4_1\sigma_2 \sigma^2_1\sigma^2_3\sigma_2 \sigma_1$. $\gamma$ has Dynnikov matrix (as given by the Dynn.exe program \cite{toby})

\begin{equation*}
\tiny D=
 \left[ \begin {array}{cccc} 17&0&12&4\\ \noalign{\medskip}28&1&20&6
\\ \noalign{\medskip}24&0&17&4\\ \noalign{\medskip}0&0&0&1\end {array}
 \right] 
\end{equation*}
\normalsize
and transition matrix

\begin{equation*}
\tiny 
T=
\left[ \begin {array}{ccc} 5&8&4\\ \noalign{\medskip}12&21&8
\\ \noalign{\medskip}12&20&9\end {array} \right]
\end{equation*}
\normalsize

associated to the invariant train track $\tau$ depicted in Figure \ref{noncomplete2}. $T$ has spectrum $\{1, 17\pm 12\sqrt{2}\}$ and  the eigenvector $v^u$ corresponding to the largest eigenvalue $\lambda=17+12\sqrt{2}$ is $(1, 1+\sqrt{2}, 1+\sqrt{2})$. That is, $$v^u=\phi_{\tau}^{-1}(\mathcal{F}^u,\mu^u)=(1, 1+\sqrt{2}, 1+\sqrt{2}).$$

\noindent Since $D_n-\tau$ contains a punctured bigon, $\tau$ is not complete and we have $\rank(\tau)=3$. Pinching across an edge of the punctured bigon gives a complete pinched train track $\tau_p$. We depict a standard embedding of $\tau_p$ with respect to the Dynnikov arcs in Figure~\ref{noncomplete21}. Since $\gamma$ fixes the prongs of $\tau$, the transition matrix associated to $\tau_p$ is given by

\begin{equation*}
\tiny T_p=
\left[ \begin {array}{cccc} 5&8&4&0\\ \noalign{\medskip}12&21&8&0
\\ \noalign{\medskip}12&20&9&0\\ \noalign{\medskip}x&y&z&1
\end {array} \right]
\end{equation*}

for some $x$, $y$, $z$ which will be determined later as $x=2$, $y=3$,   $z=1$. 

	We now compute the change of coordinate function $L:\mathcal{W}^+(\tau_p)\to \mathcal{S}_4$ in a neighbourhood of $v^u$. First observe that $\beta_1=a$, $\beta_2=b$, $\beta_3=c$. Therefore, $$b_1=\frac{a-b}{2}~~~\text{and}~~~b_2=\frac{b-c}{2}.$$ 
We have $\alpha_1=\frac{a+b}{2}$, $\alpha_3=\frac{b+c}{2}-d$. Since $\alpha_{2i-1}+\alpha_{2i}=\max(\beta_i,\beta_{i+1})$ we have, $$\alpha_2=\max(a,b)-\frac{a+b}{2}~~\text{and}~~\alpha_4=\max(b,c)-\frac{b+c}{2}+d.$$
Since $b>a$ at $v^u$, $\alpha_2=\frac{b-a}{2}$ and 

\begin{figure}
\centering
\psfrag{a}[tl]{$\scriptstyle{a}$}
\psfrag{b}[tl]{$\scriptstyle{b}$}
\psfrag{c}[tl]{$\scriptstyle{c}$}
\includegraphics[width=0.5\textwidth]{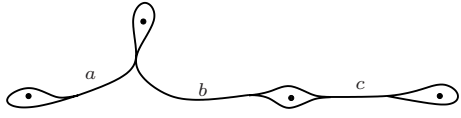}
\caption{Invariant train track for $\gamma$}\label{noncomplete2}
\end{figure}
\begin{figure}
\centering
\psfrag{a1}[tl]{$\scriptstyle{\alpha_1}$}
\psfrag{a2}[tl]{$\scriptstyle{\alpha_2}$}
\psfrag{a}[tl]{$\scriptstyle{a}$}
\psfrag{b}[tl]{$\scriptstyle{b}$}
\psfrag{c}[tl]{$\scriptstyle{c}$}
\psfrag{d}[tl]{$\scriptstyle{d}$}
\psfrag{a3}[tl]{$\scriptstyle{\alpha_3}$}
\psfrag{a4}[tl]{$\scriptstyle{\alpha_4}$}
\psfrag{b1}[tl]{$\scriptstyle{\beta_{1}}$}
\psfrag{b2}[tl]{$\scriptstyle{\beta_{2}}$}
\psfrag{b3}[tl]{$\scriptstyle{\beta_{3}}$}
\psfrag{b-d}[tl]{$\scriptstyle{b-d}$}
\psfrag{c-d}[tl]{$\scriptstyle{c-d}$}
\psfrag{(b+c)/2-d}[tl]{$\scriptstyle{\frac{b+c}{2}-d}$}
\includegraphics[width=0.6\textwidth]{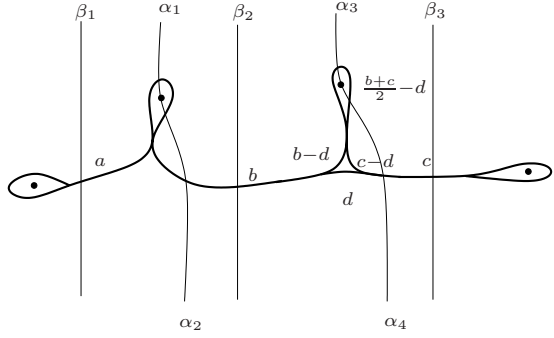}
\caption{A standard embedding of the pinched train track $\tau_p$}\label{noncomplete21}
\end{figure}

\begin{align*}
\alpha_4=\begin{cases}
\frac{b-c}{2}+d&~ b\geq c\\
\frac{c-b}{2}+d&~ b\leq c
\end{cases}
\end{align*}
 
so $a_1=-\frac{a}{2}$ and

\begin{align*}
a_2&=\frac{\alpha_4-\alpha_3}{2}=\frac{\max(b,c)}{2}-\frac{b+c}{2}+d
=\begin{cases}
d-\frac{c}{2}&b\geq c\\
d-\frac{b}{2}&b\leq c
\end{cases}
\end{align*}

Therefore, when $b\geq c$, $L:\mathcal{W}^+(\tau_p)\to \mathcal{S}_n$ is given by
\begin{equation*}
\tiny
L_1=\left[ \begin {array}{cccc} -1/2&0&0&0\\ \noalign{\medskip}0&-1/2&0&1\\ \noalign{\medskip}1/2&-1/2&0&0\\ \noalign{\medskip}0&1/2&-1/2&0
\end {array} \right], \end{equation*}
\normalsize
and when $b\leq c$, $L:\mathcal{W}^+(\tau_p)\to \mathcal{S}_n$ is given by

\begin{equation*}\tiny
L_2=\left[ \begin {array}{cccc} -1/2&0&0&0\\ \noalign{\medskip}0&0&-1/2&1\\ \noalign{\medskip}1/2&-1/2&0&0\\ \noalign{\medskip}0&1/2&-1/2&0
\end {array} \right]. \end{equation*}\normalsize
The Dynnikov matrices of $\beta$ are therefore $D_1=L_1T_pL^{-1}_1$ and $D_2=L_2T_pL^{-1}_2$. Using

\begin{equation*}\tiny T_p=
\left[ \begin {array}{cccc} 5&8&4&0\\ \noalign{\medskip}12&21&8&0
\\ \noalign{\medskip}12&20&9&0\\ \noalign{\medskip}x&y&z&1
\end {array} \right],
\end{equation*}\normalsize

we compute that, for both $i=1$ and $i=2$, $D_i$ is given by

\begin{equation*}
D_i=\tiny{\left[ \begin {array}{cccc} 17&0&12&4\\ \noalign{\medskip}40-2\,x-2\,y-2\,z&1&28-2\,y-2\,z&8-2\,z\\ \noalign{\medskip}24&0&17&4
\\ \noalign{\medskip}0&0&0&1\end {array} \right]}.\end{equation*}

\normalsize 
\begin{figure}
\centering
\psfrag{1}[tl]{$\scriptstyle{\tau_1}$}
\psfrag{2}[tl]{$\scriptstyle{\tau_2}$}
\psfrag{a}[tl]{$\scriptstyle{a}$}
\psfrag{b}[tl]{$\scriptstyle{b}$}
\psfrag{c}[tl]{$\scriptstyle{c}$}
\includegraphics[width=0.5\textwidth]{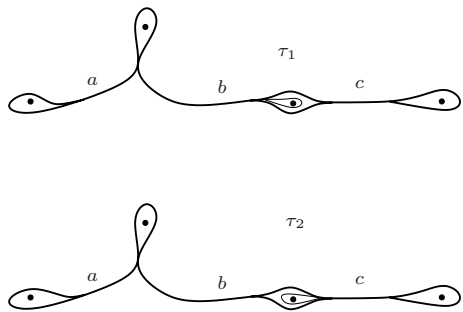}
\caption{Diagonal extensions of $\tau$}\label{noncomplete22}
\end{figure}

Therefore, there is a unique Dynnikov matrix $D$ as expected by Theorem \ref{uniquematrix}. Comparing the second row of $D_i$ with the known Dynnikov matrix $D$ gives $x=2$, $y=3$, $z=1$ as claimed above. The spectrum of $T_p$ and $D$ is $$\{1,1, 17\pm 12\sqrt{2}\}.$$

Figure \ref{noncomplete22} depicts the two possible diagonal extensions $\tau_1$ and $\tau_2$ of $\tau$. Observe that, 
given $(\mathcal{F},\mu)\in \mathcal{MF}(\tau_p)$, $(\mathcal{F},\mu)\in \mathcal{MF}(\tau_1)$ if and only if $ b\geq c$, and $(\mathcal{F},\mu)\in \mathcal{MF}(\tau_2)$ if and only if $ c\geq b$, corresponding to the two linear regions in the above coordinate change. 

\end{example}

\begin{quest}\label{quest}
Let $\beta\in B_n$ be a pseudo\,-Anosov  braid with unstable invariant foliation $(\mathcal{F}^u,\mu^u)$, dilatation $\lambda>1$ and regular invariant train track $\tau$ having transition matrix $T$. If $\beta$ permutes the prongs of $(\mathcal{F}^u,\mu^u)$ non-trivially, is every Dynnikov matrix $D_i$ isospectral to $T$ up to roots of unity? 
\end{quest}
The claim in Question \ref{quest} has been confirmed with a wide range of examples. The difficulty which arises in this case is explained in Remark \ref{sonornek}.
\begin{remark}\label{sonornek}
Let $\tau_i$ be the diagonal extensions of $\tau$ $1\leq i \leq \xi$. When $\beta$ permutes the singularities of $(\mathcal{F}^u,\mu^u)$ non-trivially  we have $\beta(\mathcal{MF}(\tau_i))=\mathcal{MF}(\tau_{i+1})$ for each $1\leq i \leq \xi$ and the induced action  $\beta:\mathcal{W}(\tau_i)\to\mathcal{W}(\tau_{i+1})$ is given by the matrix $$\tilde{T}=\left[ \begin {array}{cc} T&0\\ \noalign{\medskip}0&{\it Id}
\end {array} \right]
$$
with respect to the above choice of bases of $\mathcal{W}(\tau_i)$ and $\mathcal{W}(\tau_{i+1})$. We therefore have the commutative diagram for each $i$:

\begin{align*}\label{diagramcomplete}
\begin{CD}
\mathcal{W}^{+}(\tau_i) @>\tilde{T}>>\mathcal{W}^{+}(\tau_{i+1})\\
@V\phi_{\tau_i}  VV @V\phi_{\tau_{i+1}} VV\\
\mathcal{MF}(\tau_i) @>\beta>>\mathcal{MF}(\tau_{i+1})\\
@V \rho VV @V \rho VV\\
\mathcal{S}_n @>F>>\mathcal{S}_n\\
\end{CD}
\end{align*}
where $L_i=\rho\circ\phi_{\tau_i}$ is linear in a neighbourhood $U_i$ of $v^u$ in $\mathcal{W}^+(\tau_i)$. For $1\leq i\leq \xi$, let $\mathcal{R}_i=L_i(U_i)$ and $(a^u,b^u)$ denote the Dynnikov coordinates of $(\mathcal{F}^u,\mu^u)$. Then, $\displaystyle\bigcup_{1\leq i\leq \xi}\mathcal{R}_i$ is a neighbourhood of $(a^u,b^u)$, and in $R_i$ the Dynnikov matrices are given by $D_i=F|_{\mathcal{R}_i}=L_{i+1}\circ \tilde{T}\circ L_i^{-1}$. Therefore the spectrum of $D_i$ ($1\leq i\leq \xi$) for $\beta\in B_n$ does not have an obvious interpretation. 

Note that when $\beta$ permutes the prongs of $(\mathcal{F}^u,\mu^u)$ non-trivially, then for some $m\in \mathbb{Z}^{+}$, $\beta^m$ fixes the prongs. The transition matrix for $\beta^m$ on a diagonal extension of $\tau$ is of the form

$$
T'=\left[ \begin {array}{cc} T^m&0\\ \noalign{\medskip}0&Id\end {array}
 \right]. 
$$

By Theorem \ref{uniquematrix}  the Dynnikov matrices for $\beta^m$ are the same and isospectral to $T^m$ up to some eigenvalues $1$. 
\end{remark}

\section{A comparison with known algorithms}\label{comparemethods}
 The main reasons that our method works much faster than the train track approach are: first, because of the simple way we put global coordinates on $\mathcal{MF}_n$; and  second, because it is easy to find the Dynnikov coordinates of $[\mathcal{F}^u,\mu^u]$ on $\mathcal{PMF}_n$ numerically since it is a globally attracting fixed point of the induced action. In addition, the method is more transparent since it relies on algebraic calculations rather than on understanding the image of a train track under the action of an isotopy class.  We encourage the reader to take a random braid and try the two different methods using the train track and Dynnikov program \cite{toby}.

To give an explicit example, let $\beta$ be the $4$-braid 
$$\sigma^{-1}_1\sigma^{-3}_2\sigma^{-5}_3\sigma^4_1\sigma^{-2}_2 \sigma^{-1}_3\sigma_1\sigma_2\sigma^{-2}_3
(\sigma_2 \sigma^{-2}_3)^{19}\sigma^{-8}_1\sigma^{-1}_3\sigma^{-2}_1\sigma^2_2\sigma^{-1}_3
\sigma^{-1}_1\sigma_2\sigma_3\sigma_1\sigma^{-1}_2\sigma^{-1}_3.$$
Using Dynn.exe program in \cite{toby} we find that $\beta$ has a unique Dynnikov matrix given by 
\small{$$\left[ \begin {array}{cccc} -68900596045753&200002959211464&

146825523685804&-943752747512\\ \noalign{\medskip}-181490417757959&

526825930446403&386751743244292&-2485930314639\\ \noalign{\medskip}-

188609831321041&547491989409364&401923043417627&-2583447121425

\\ \noalign{\medskip}76020009608848&-220669018174468&-161996823859176&

1041269554295\end {array} \right].$$}
which has dilatation approximately $8.6\times 10^{14}$ (entropy is $\sim 34.38$). This means the initial image edge paths have length approximately $8.6\times 10^{14}$. An edge path of this length occupies approximately $10^5$ GB of memory and hence the train track program can not even start, whereas its Dynnikov matrix is found in less than a second.

\begin{remark}
 We note that providing explicit bounds on the complexity of the algorithm used to compute Dynnikov matrices is difficult, since we don't know the size of the Dynnikov regions and hence how many iterates are needed to land in it. 

\end{remark}

Finally, we remark that the method developed in this paper and in \cite{paper1} can be realized on any compact, orientable surface using for instance the well known Dehn-Thurston coordinates \cite{penner}. The main reason we use the Dynnikov coordinates is that it gives us a much easier calculational approach and so is particularly suitable for our problems on the finitely punctured disk. However, it would be very interesting to relate Dynnikov coordinates to Dehn-Thurston coordinates and generalize our results to higher genus surfaces.

\normalsize

 \section*{Acknowledgements} The author would like to thank Toby Hall for his guidance and critical reading 
during the preparation of this paper. She would also like to thank Philip Boyland, Anthony Manning, Lasse Rempe and the referees for their constructive suggestions and valuable comments.

\bibliographystyle{plainnat}
\bibliography{myrefsaa}

\bibliographystyle{plainnat}

\medskip
% The data information below will be filled by AIMS editorial staff
Received xxxx 20xx; revised xxxx 20xx.
\medskip

\end{document}